\documentclass[a4paper,reqno,10pt]{amsart}
\makeatletter
\newcommand*{\rom}[1]{\expandafter\@slowromancap\romannumeral #1@}
\makeatother
\usepackage{epsf,graphicx,epsfig}
\usepackage{amscd}
\usepackage{amsmath,latexsym,amssymb,amsthm}
\usepackage[nospace,noadjust]{cite}
\usepackage{textcomp}
\usepackage{setspace,cite}
\usepackage{mathrsfs,amsmath}
\usepackage{lscape,fancyhdr,fancybox}
\usepackage{stmaryrd}
\usepackage[all,cmtip]{xy}
\usepackage{tikz-cd}
\usepackage{cancel}
\usetikzlibrary{shapes,arrows,decorations.markings}
\usepackage{graphicx}

\usepackage{color}
\usepackage{bbm, dsfont}
\usepackage{xcolor}
\usepackage{url}
\usepackage{enumerate}
\usepackage[mathscr]{euscript}

  \theoremstyle{plain}
\swapnumbers
    \newtheorem{thm}{Theorem}[section]
    \newtheorem{proposition}[thm]{Proposition}

    \newtheorem{subsec}[thm]{}
\theoremstyle{definition}
    \newtheorem{definition}[thm]{Definition}
        \newtheorem{remark}[thm]{Remark}
    \newtheorem{exam}[thm]{Example}

\theoremstyle{remark}

\title{}
\author{}
\date{}
\usepackage{amssymb}

\usepackage{hyperref}
\hypersetup{
colorlinks,
citecolor=blue,
filecolor=black,
linkcolor=blue,
urlcolor=black
}

\begin{document}

\title{Homotopification and categorification of Leibniz conformal algebras}

\author{Apurba Das}
\address{Department of Mathematics,
Indian Institute of Technology, Kharagpur 721302, West Bengal, India.}
\email{apurbadas348@gmail.com, apurbadas348@maths.iitkgp.ac.in}

\author{Anupam Sahoo}
\address{Department of Mathematics,
Indian Institute of Technology, Kharagpur 721302, West Bengal, India.}
\email{anupamsahoo23@gmail.com}

\begin{abstract}
Bakalov, Kac and Voronov introduced Leibniz conformal algebras (and their cohomology) as a non-commutative analogue of Lie conformal algebras. Leibniz conformal algebras are closely related to field algebras which are non-skew-symmetric generalizations of vertex algebras. In this paper, we first introduce $Leib_\infty$-conformal algebras (also called strongly homotopy Leibniz conformal algebras) where the Leibniz conformal identity holds up to homotopy. We give some equivalent descriptions of $Leib_\infty$-conformal algebras and characterize some particular classes of  $Leib_\infty$-conformal algebras in terms of the cohomology of Leibniz conformal algebras and crossed modules of Leibniz conformal algebras. On the other hand, we also introduce Leibniz conformal $2$-algebras that can be realized as the categorification of Leibniz conformal algebras. Finally, we observe that the category of Leibniz conformal $2$-algebras is equivalent to the category of $2$-term $Leib_\infty$-conformal algebras.
\end{abstract}

\maketitle



\medskip

\begin{center}



\end{center}



\thispagestyle{empty}

\tableofcontents

\section{Introduction}

\subsection{Origin of conformal algebras} After the seminal work of Belavin, Polyakov and Zamolodchikov \cite{bpz}, conformal field theory has become a fundamental subject with many remarkable connections to mathematics and physics. A rigorous description of chiral algebras (also called vertex algebras by mathematicians) in conformal field theory was proposed by Borcherds \cite{bor} and subsequently studied in \cite{kac, dong, fhl, flm}. The study of Lie conformal algebras was initiated by Kac \cite{kac} in view of their intimate connections with $2$-dimensional quantum field theory, vertex algebras and infinite dimensional Lie algebras. Namely, a Lie conformal algebra encodes the singular part of the operator product expansion of chiral fields in conformal field theory. To some extent, Lie conformal algebras are related to vertex algebras in the same way Lie algebras are related to their universal enveloping algebras. Structures of finite and simple conformal algebras were extensively studied in \cite{andrea-kac,kolesnikov,guo-tan}. On the other hand, in the study of representations of Lie conformal algebras, the authors in \cite{bkv} introduced the notion of associative conformal algebras. It is important to remark that representations and cohomology of finite Lie and associative conformal algebras have significant development in the last twenty-five years, see \cite{bkv,kolesnikov-kozlov,sole-kac} and the references therein.

\subsection{Higher structures: homotopy algebras and categorifications} Higher structures, such as homotopy algebras (also called $\infty$-algebras) and categorifications of algebras play important roles in higher Lie theory, higher gauge theory and conformal field theory. Note that, higher structures appear when we increase the flexibility of an algebraic structure. This can be done in two ways, namely, by homotopification and categorification of given algebraic identity/identities. Homotopy algebras are precisely homotopification of algebras. They can be realized as homotopy invariant extensions of differential graded algebras. The notion of a strongly homotopy Lie algebra or an $L_\infty$-algebra is the homotopification of Lie algebra \cite{lada-markl}, i.e. Jacobi identity holds up to homotopy. They play a prominent role in deformations of algebraic structures \cite{inf-def}, quantization of Poisson manifolds \cite{kontsevich} and higher symplectic geometry \cite{rogers}. On the other hand, $2$-algebras are categorifications of algebras. They are obtained when we replace sets (resp. maps, equalities) with categories (resp. functors, natural isomorphisms). Unlike ordinary algebras (where we consider two maps to be equal), here we consider the corresponding functors to be naturally isomorphic. The notion of Lie $2$-algebras was introduced by Baez and Crans \cite{baez-crans} as the categorification of Lie algebras. They are closely related to the Zamolodchikov tetrahedron equation which can be considered as the categorification of the set-theoretical solutions of the Yang-Baxter equation. It is important to remark that the concept of categorification leads to many surprising results in topological field theory and string theory. It has been observed in \cite{baez-crans} that homotopy algebras and $2$-algebras are closely connected. More precisely, they showed that the category of $2$-term $L_\infty$-algebras and the category of Lie $2$-algebras are equivalent. Among others, they studied skeletal and strict $L_\infty$-algebras and gave their characterizations.

Homotopy algebras and $2$-algebras can be generalized to other types of algebraic structures. The notion of a Leibniz algebra (first introduced by Bloh \cite{bloh} and popularized by Loday \cite{loday}) is a non-skew-symmetric generalization of a Lie algebra. Many results that are true for Lie algebras can be generalized to Leibniz algebras. In \cite{ammar-poncin} Ammar and Poncin first introduced the notion of a $Leib_\infty$-algebra in their study of graded free zinbiel coalgebra. A more concrete description of a $Leib_\infty$-algebra is given in \cite{khuda} (see also \cite{xu,uchino}). On the other hand, Leibniz $2$-algebras are considered and their relations with $Leib_\infty$-algebras are discussed in \cite{sheng-liu}.

\subsection{Layout of the paper} Like Leibniz algebras are a non-skew-symmetric analogue of Lie algebras, the concept of Leibniz conformal algebras introduced in \cite{bkv} forms a non-skew-symmetric analogue of Lie conformal algebras. Leibniz conformal algebras are closely related to field algebras \cite{baka-kac}, which are non-commutative generalizations of vertex algebras. It has been proved by Zhang \cite{zhang} that the category of Leibniz conformal algebras is equivalent to the category of equivalence classes of formal distribution Leibniz algebras. Representations and cohomology of Leibniz conformal algebras are studied extensively in \cite{bkv, zhang}. Quadratic Leibniz conformal algebras are considered by Zhou and Hong \cite{zhou-hong}. In \cite{hong-yuan} the authors defined a unified product (that simultaneously generalizes twisted product, crossed product and bicrossed product) of Leibniz conformal algebras and introduced some cohomological objects to classify such unified products.

Our main aim in this paper is to study homotopification and categorification of Leibniz conformal algebras. Namely, we provide two generalizations of Leibniz conformal algebras by extending flexibility to the Leibniz conformal identity.

\medskip

(i) At first, we introduce {\em $Leib_\infty$-conformal algebras}  (also called {\em strongly homotopy Leibniz conformal algebras}) in which the Leibniz conformal identity holds up to homotopy. We observe that $Leib_\infty$-conformal algebras can be better understood when we shift the underlying graded space. Using this observation, we give a Maurer-Cartan characterization of $Leib_\infty$-conformal algebras. More precisely, given a graded $\mathbb{C}[\partial]$-module $\mathcal{G}$, we construct a graded Lie algebra whose Maurer-Cartan elements correspond to $Leib_\infty$-conformal algebra structures on $\mathcal{G}$. Our graded Lie algebra generalizes Balavoine's graded Lie algebra \cite{bala} that characterizes classical Leibniz algebras. As a consequence of our characterization, we define the cohomology of a $Leib_\infty$-conformal algebra. Then we put special attention to those $Leib_\infty$-conformal algebras whose underlying graded $\mathbb{C} [\partial]$-module $\mathcal{G}$ is concentrated in arity $0$ and $1$. We call them {\em $2$-term $Leib_\infty$-conformal algebras}. The collection of all $2$-term $Leib_\infty$-conformal algebras and homomorphisms between them forms a category, denoted by ${\bf 2Leib_\infty Conf}$. Motivated by the results of \cite{baez-crans}, we consider skeletal and strict $Leib_\infty$-conformal algebras. Among others, we show that skeletal algebras are characterized by third cocycles of Leibniz conformal algebras and strict algebras are characterized by crossed modules of Leibniz conformal algebras.

\medskip

(ii) In the next, we consider {\em Leibniz conformal $2$-algebras} as categorification of Leibniz conformal algebras. More precisely, a Leibniz conformal $2$-algebra is a category $C$ (internal to the category of $\mathbb{C}[\partial]$-modules) endowed with a $\mathbb{C}$-linear conformal sesquilinear functor $[\cdot_\lambda \cdot ]: C \otimes C \rightarrow C [[ \lambda ]]$ (that need not satisfy the Leibniz conformal identity) and a conformal Leibnizator that satisfy some suitable identity which can be described by the commutativity of the diagram (\ref{hexa}). We show that the collection of all Leibniz conformal $2$-algebras and homomorphisms between them forms a category, denoted by {\bf LeibConf2}. We also show that the category {\bf LeibConf2} is equivalent to the category ${\bf 2Leib_\infty Conf}$ of $2$-term $Leib_\infty$-conformal algebras. Finally, we give an example of a Leibniz conformal $2$-algebra associated with any $2$-term complex of finite $\mathbb{C}[\partial]$-modules.

The results of the present paper can be summarized in the following schematic diagram:

\medskip

\medskip

\medskip

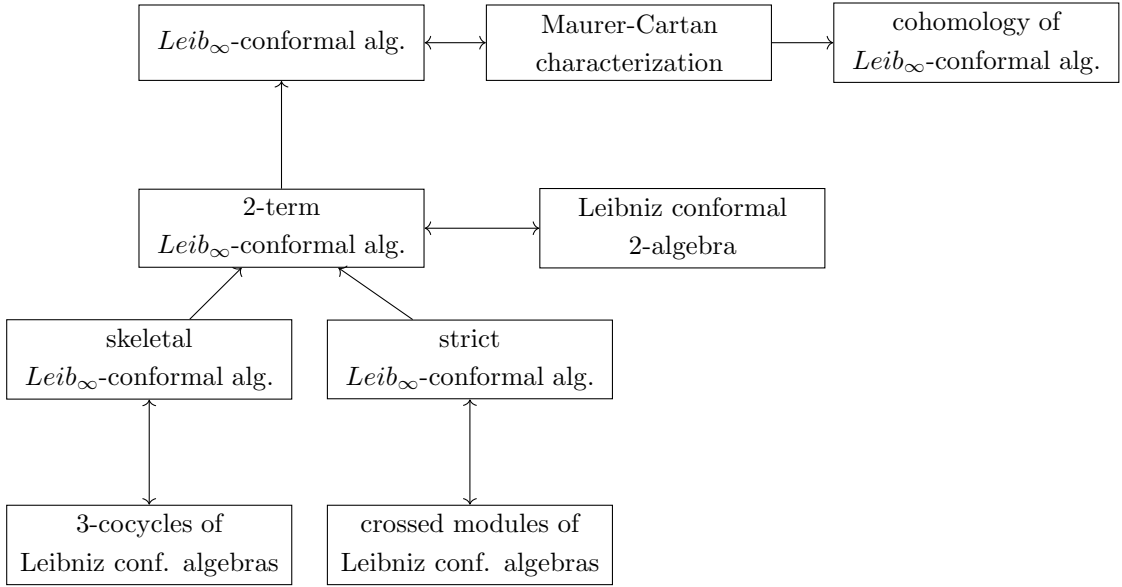
\tikzstyle{block}=[draw,rectangle,text width=10em,text centered, minimum height=10mm, minimum width= 4mm, node distance=7em]
\begin{center}
\begin{tikzpicture}[
implies/.style={double,double equal sign distance,-implies},]
\node[block](lca){$Leib_\infty$-conformal alg.};
\node[block, below of =lca](2term){$2$-term \\ $Leib_\infty$-conformal alg.};
\node[block, right of =lca, xshift=6em](mc){Maurer-Cartan\\ characterization};
\node[block, right of =mc, xshift=6em](co){cohomology of \\  $Leib_\infty$-conformal alg.};
\node[block, right of =2term, xshift=8em](Leibconf2){Leibniz conformal \\ $2$-algebra};
\node[block, below left of =2term](skeletal){skeletal \\ $Leib_\infty$-conformal alg.};
\node[block, right of =skeletal, xshift=5em](strict){strict  \\ $Leib_\infty$-conformal alg.};
\node[block, below of =skeletal](3co){$3$-cocycles of \\ Leibniz conf. algebras};
\node[block, below of =strict](crossed){crossed modules of \\ Leibniz conf. algebras};
\draw [<->] (mc) to (lca);
\draw [<->] (2term) to (Leibconf2);
\draw [->] (2term) to (lca);
\draw [->] (mc) to (co);
\draw [->] (skeletal) to (2term);
\draw [->] (strict) to (2term);
\draw [<->] (skeletal) to (3co);
\draw [<->] (strict) to (crossed);
\end{tikzpicture}
\end{center}

\medskip

\noindent where an arrow $\framebox{A} \rightarrow \framebox{B}$ means that either $A$ implies $B$ or $A$ is a subcollection of $B.$

\medskip

\medskip


\subsection{Organization of the paper} In Section \ref{sec2}, we recall Leibniz conformal algebras and their cohomology with coefficients in a representation. We introduce and study $Leib_\infty$-conformal algebras in Section \ref{sec3}. Among others, we give the Maurer-Cartan characterization of $Leib_\infty$-conformal algebras. Section \ref{sec4} is devoted to study $2$-term $Leib_\infty$-conformal algebras. In particular, we characterize skeletal and strict $Leib_\infty$-conformal algebras. Finally, in Section \ref{sec5}, we consider Leibniz conformal $2$-algebras and show that the category {\bf LeibConf2} is equivalent to the category ${\bf 2Leib_\infty Conf}$.


\section{Background on Leibniz conformal algebras}\label{sec2}
In this section, we recall some definitions and basic facts about Leibniz conformal algebras. In particular, we recall representations and cohomology of Leibniz conformal algebras. Our main references are \cite{bkv,hong-yuan, zhang}.

Leibniz algebras are a non-skew-symmetric analogue of Lie algebras that were first considered by Bloh \cite{bloh} and rediscovered by Loday \cite{loday}.

\begin{definition}
    A {\em (left) Leibniz algebra} is a vector space $\mathfrak{g}$ equipped with a bilinear map (called the bracket) $[~,~] : \mathfrak{g} \times \mathfrak{g} \rightarrow \mathfrak{g}, ~(x, y) \mapsto [x, y]$  satisfying the left Leibniz identity:
    \begin{align}\label{ll-identity}
        [x, [y, z]] = [[x, y], z] + [y, [x,z]], \text{ for } x, y, z \in \mathfrak{g}.
    \end{align}
\end{definition}

The concept of Leibniz conformal algebras was first considered by Bakalov, Kac and Voronov \cite{bkv} as the non-skew-symmetric analogue of Lie conformal algebras. They can also be considered as the conformal analogue of Leibniz algebras.

\begin{definition}
    A {\em Leibniz conformal algebra} is a $\mathbb{C}[\partial]$-module $\mathfrak{g}$ equipped with a $\mathbb{C}$-linear bracket (called the $\lambda$-bracket)
    \begin{align*}
        [ \cdot_\lambda \cdot] : \mathfrak{g} \otimes \mathfrak{g} \rightarrow \mathfrak{g} [[ \lambda]], ~ x \otimes y \mapsto [x_\lambda y]
    \end{align*}
    that satisfies the conformal sesquilinearity conditions
    \begin{align}\label{cs-identity}
        [\partial x _\lambda y] = - \lambda [x_\lambda y], \quad [x_\lambda \partial y] = (\partial + \lambda ) [x_\lambda y],
    \end{align}
    and the following Leibniz conformal identity:
    \begin{align}\label{lc-identity}
        [x_\lambda [y_\mu z]] = [[x_\lambda y]_{\lambda + \mu} z] + [y_\mu [x_\lambda z]], \text{ for } x, y, z \in \mathfrak{g}.
    \end{align}
\end{definition}

\medskip

A Leibniz conformal algebra as above may be denoted by the pair $(\mathfrak{g}, [ \cdot_\lambda \cdot])$ or simply by $\mathfrak{g}$ if the $\lambda$-bracket is clear from the context. Since the Leibniz conformal identity (\ref{lc-identity}) is a generalization of the left Leibniz identity (\ref{ll-identity}), a Leibniz conformal algebra as defined above is often called a left Leibniz conformal algebra. 

Let $M$ be a $\mathbb{C} [\partial]$-module equipped with a $\lambda$-bracket. A $\mathbb{C}$-linear map $f : M \rightarrow M [[\lambda]], ~ m \mapsto f_\lambda (m)$ is said to be a {\em conformal derivation} if $f_\lambda [m_\mu n] = [(f_\lambda m)_{\lambda + \mu} n] + [m_\mu (f_\lambda n)]$, for $m,n \in M$. It follows that the identity (\ref{lc-identity}) in a Leibniz conformal algebra is equivalent to the fact that all the left translations $[x_\lambda -] : \mathfrak{g} \rightarrow \mathfrak{g} [[\lambda]]$ are conformal derivations for the $\lambda$-bracket.

The concept of a right Leibniz conformal algebra can be defined similarly.
Throughout the paper, by a Leibniz conformal algebra, we shall always mean a left Leibniz conformal algebra. However, all the results about left Leibniz conformal algebras can be easily generalized to right Leibniz conformal algebras without much work.

\begin{remark}
    A Lie conformal algebra is a Leibniz conformal algebra $(\mathfrak{g}, [\cdot_\lambda \cdot])$ whose $\lambda$-bracket is skew-symmetric (i.e. $[x_\lambda y] = - [y_{-\partial - \lambda} x]$, for all $x, y \in \mathfrak{g}$) and the image of the $\lambda$-bracket lies in the space of $\mathfrak{g}$-valued polynomials in $\lambda$ \cite{bkv}. Thus, Lie conformal algebras are natural examples of Leibniz conformal algebras. Representations of Lie conformal algebras also induce Leibniz conformal algebras. Recall from \cite{bkv} that a representation of a Lie conformal algebra $(\mathfrak{g}, [\cdot_\lambda \cdot])$ is a $\mathbb{C}[\partial]$-module $M$ equipped with a $\mathbb{C}$-linear map (called the $\lambda$-action) $\cdot_\lambda \cdot : \mathfrak{g} \otimes M \rightarrow M[\lambda]$, $x \otimes v \mapsto x_\lambda v$ that satisfies $(\partial x)_\lambda v = - \lambda x_\lambda v$, $x_\lambda (\partial v) = (\partial + \lambda) x_\lambda v$ and 
    \begin{align*}
        (x_\lambda y)_{\lambda + \mu} v = x_\lambda (y_\mu v) - y_\mu (x_\lambda v), \text{ for all }x, y \in \mathfrak{g}, v \in M.
    \end{align*}
    If $(\mathfrak{g}, [\cdot_\lambda \cdot])$ is a Lie conformal algebra and $M$ is a representation, then the graded $\mathbb{C}[\partial]$-module $\mathfrak{g} \oplus M$ inherits a Leibniz conformal algebra structure with the $\lambda$-bracket
    \begin{align*}
        \{  (x,u)_\lambda (y, v) \} = ([x_\lambda y], x_\lambda v), \text{ for } (x, u), (y, v) \in \mathfrak{g} \oplus M.
    \end{align*}
\end{remark}

\medskip

Let $(\mathfrak{g}, [\cdot_\lambda \cdot])$ be a Leibniz conformal algebra. For each $j \geq 0$, we define the $j$-th product on $\mathfrak{g}$ by a $\mathbb{C}$-linear map $\cdot_{(j)} \cdot : \mathfrak{g} \otimes \mathfrak{g} \rightarrow \mathfrak{g}$ that satisfies
\begin{align*}
    [x_\lambda y] = \sum_{j \geq 0} \frac{\lambda^j}{ j !} \big( x_{(j)} y  \big), \text{ for all } x, y \in \mathfrak{g}. 
\end{align*}
Since $[x_\lambda y]$ is a $\mathfrak{g}$-valued formal power series in $\lambda$ with possibly infinitely many terms, there are infinitely many $j$'s for which $x_{(j)} y \neq 0$. In terms of the $j$-th products, the identities (\ref{cs-identity}) and (\ref{lc-identity}) are equivalent to the followings: for all $x, y, z \in \mathfrak{g}$ and $j, m, n \geq 0$,
\begin{align*}
    (\partial x)_{(j)} y = - j x_{(j-1)} y, \quad x_{(j)} (\partial y) = \partial ( x_{(j)} y) + j x_{(j-1)} y, \\
    x_{(m)} (y_{(n)} z ) = \sum_{k=0}^m \binom{m}{k} (x_{(k)} y)_{(m+n-k)} z + y_{(n)} (x_{(m)} z).
\end{align*}

\begin{remark}
Note that the concept of Leibniz conformal algebras can be generalized to the graded context. More precisely, a {\em graded Leibniz conformal algebra} is a pair $(\mathcal{G} = \oplus_{i \in \mathbb{Z}} \mathcal{G}_i, [\cdot_\lambda \cdot])$ of a graded $\mathbb{C}[\partial]$-module $\mathcal{G} = \oplus_{i \in \mathbb{Z}} \mathcal{G}_i$ with a degree $0$ conformal sesquilinear map $[\cdot_\lambda \cdot] : \mathcal{G} \otimes \mathcal{G} \rightarrow \mathcal{G} [[ \lambda ]]$ that satisfies the graded version of the identity (\ref{lc-identity}), namely,
\begin{align*}
     [x_\lambda [y_\mu z]] = [[x_\lambda y]_{\lambda + \mu} z] + (-1)^{|x| |y|} [y_\mu [x_\lambda z]], \text{ for homogeneous } x, y, z \in \mathcal{G}
\end{align*}
It is further called a {\em differential graded Leibniz conformal algebra} if there exists a degree $-1$ $\mathbb{C} [\partial]$-linear map $d: \mathcal{G} \rightarrow \mathcal{G}$ that satisfies $d^2 = 0$ and $d [x_\lambda y]  = [(dx)_\lambda y] + (-1)^{|x|} [x_\lambda (dy)]$, for all $x, y \in \mathcal{G}$.
\end{remark}

\begin{definition}
    Let $(\mathfrak{g}, [\cdot_\lambda \cdot])$ be a Leibniz conformal algebra. A {\em representation} of $(\mathfrak{g}, [\cdot_\lambda \cdot])$ is given by a $\mathbb{C}[\partial]$-module $M$ equipped with two $\mathbb{C}$-linear maps (called the left and right $\lambda$-actions, respectively)
    \begin{align*}
        \cdot_\lambda \cdot : \mathfrak{g} \otimes M \rightarrow M [[ \lambda]], \! \!  ~ x \otimes v \mapsto x_\lambda v \quad \text{ and } \quad \cdot_\lambda \cdot : M \otimes \mathfrak{g} \rightarrow M [[\lambda]],  \! \! ~ v \otimes x \mapsto v_\lambda x
    \end{align*}
    satisfying the following set of identities: for any $x, y \in \mathfrak{g}$ and $v \in M$, 
    \begin{align*}
    (\partial x)_\lambda v &= - \lambda (x_\lambda v), ~~~~ x_\lambda (\partial v) = (\partial + \lambda) x_\lambda v, \\
    (\partial v)_\lambda x &= - \lambda (v_\lambda x), ~~~~ v_\lambda (\partial x) = (\partial + \lambda) v_\lambda x, \\
&x_\lambda (y_\mu v) = [x_\lambda y]_{\lambda + \mu} v + y_\mu (x_\lambda v), \\
&x_\lambda (v_\mu y) = (x_\lambda v)_{\lambda + \mu} y + v_\mu [x_\lambda y],\\
&v_\lambda [x_\mu y] = (v_\lambda x)_{\lambda + \mu} y + x_\mu (v_\lambda y).
\end{align*}
\end{definition}

We denote a representation as above simply by $M$ when the left and right $\lambda$-actions are clear from the context.
It follows from the above definition that any Leibniz conformal algebra $(\mathfrak{g}, [\cdot_\lambda \cdot])$ can be regarded as a representation of itself, where both the left and right $\lambda$-actions are given by the $\lambda$-bracket. This is called the {\em adjoint representation}.

Let $\mathfrak{g}$ and $M$ be two $\mathbb{C}[\partial]$-modules (not necessarily equipped with any additional structures). For any $n \geq 1$, a $\mathbb{C}$-linear map
\begin{align*}
\varphi : \mathfrak{g}^{\otimes n} \rightarrow M [[ \lambda_1, \ldots \lambda_{n-1} ]], \! \! ~  \varphi (x_1, \ldots, x_n) \mapsto \varphi_{  \lambda_1, \ldots \lambda_{n-1}  } (x_1, \ldots, x_n)
\end{align*}
is said to be {\em conformal sesquilinear} if the following conditions are hold: for all $x_1, \ldots, x_n \in \mathfrak{g}$,
\begin{align*}
    \varphi_{  \lambda_1, \ldots \lambda_{n-1}  } (x_1, \ldots, \partial x_i, \ldots,  x_n) = - \lambda_i \!  ~\varphi_{  \lambda_1, \ldots \lambda_{n-1}  } (x_1, \ldots, x_n), \text{ for } 1 \leq i \leq n-1,\\
    \varphi_{  \lambda_1, \ldots \lambda_{n-1}  } (x_1, \ldots, x_{n-1}, \partial x_n) = (\partial + \lambda_1 + \cdots + \lambda_{n-1})  \!  ~ \varphi_{  \lambda_1, \ldots \lambda_{n-1}  } (x_1, \ldots, x_n).
\end{align*}
We denote the space of all such conformal sesquilinear maps by $\mathrm{Hom}_{cs} (\mathfrak{g}^{\otimes n}, M [[ \lambda_1, \ldots, \lambda_{n-1}]]).$ Note that, for $n=1$, we have $\mathrm{Hom}_{cs} (\mathfrak{g}, M) = \mathrm{Hom}_{ \mathbb{C}[\partial]} (\mathfrak{g}, M)$ the space of all $\mathbb{C}[\partial]$-linear maps from $\mathfrak{g}$ to $M$.

\medskip

Let $(\mathfrak{g}, [\cdot_\lambda \cdot])$ be a Leibniz conformal algebra and $M$ be a representation. For each $n \geq 0$, we define the space $C^n_\mathrm{cLeib} (\mathfrak{g}, M)$ of $n$-cochains by
\begin{align*}
    C^n_\mathrm{cLeib} (\mathfrak{g}, M) = \begin{cases}
M/ \partial M & \text{ if } n =0,\\
\mathrm{Hom}_{cs} (\mathfrak{g}^{\otimes n}, M [[\lambda_1, \ldots, \lambda_{n-1}]]) & \text{ if } n \geq 1.
    \end{cases}
\end{align*}
Then there is a map $\delta : C^n_\mathrm{cLeib} (\mathfrak{g}, M) \rightarrow  C^{n+1}_\mathrm{cLeib} (\mathfrak{g}, M) $ given by
\begin{align*}
    \delta (v + \partial M) (x) =~& (- v_\lambda x)|_{ \lambda = 0}, \text{ for } v + \partial M \in M / \partial M \text{ and } x \in \mathfrak{g}, \\
    (\delta \varphi)_{\lambda_1, \ldots, \lambda_n} (x_1, \ldots, x_{n+1}) =~& \sum_{i=1}^n (-1)^{i+1} ~ {x_i}_{\lambda_i} \big(  \varphi_{\lambda_1, \ldots, \widehat{\lambda_i}, \ldots, \lambda_n}  (x_1, \ldots, \widehat{x_i}, \ldots, x_{n+1}) \big) \\
   ~& + (-1)^{n+1} \big(  \varphi_{\lambda_1, \ldots, \lambda_{n-1}} (x_1, \ldots, x_n)  \big)_{\lambda_1 + \cdots + \lambda_n} x_{n+1} \\
    + \sum_{1 \leq i < j \leq n+1} (-1)^i ~& \varphi_{ \lambda_1, \ldots, \widehat{\lambda_i}, \ldots, \lambda_{j-1} , \lambda_i + \lambda_j, \ldots, \lambda_n } (x_1, \ldots, \widehat{x_i}, \ldots, x_{j-1}, [{x_i}_{\lambda_i} x_j], \ldots, x_{n+1}),
\end{align*}
for $\varphi \in C^{n \geq 1}_\mathrm{cLeib} (\mathfrak{g}, M)$ and $x_1, \ldots, x_{n+1} \in \mathfrak{g}$. It has been shown in \cite{zhang} that $\delta^2 = 0$. In other words $\{ C^\bullet_\mathrm{cLeib} (\mathfrak{g}, M), \delta \}$ is a cochain complex. The corresponding cohomology groups are called the {\em cohomology} of the Leibniz conformal algebra $(\mathfrak{g}, [ \cdot_\lambda \cdot])$ with coefficients in the representation $M$.



\section{Strongly homotopy Leibniz conformal algebras}\label{sec3}
The notion of $Leib_\infty$-algebras was first introduced by Ammar and Poncin \cite{ammar-poncin} in the operadic study of Leibniz algebras (see also \cite{khuda}). In this section, we introduce the concept of $Leib_\infty$-conformal algebras as the homotopification of Leibniz conformal algebras. Given a graded $\mathbb{C}[\partial]$-module $\mathcal{G}$, we construct a graded Lie algebra whose Maurer-Cartan elements correspond to $Leib_\infty$-conformal algebra structures on $\mathcal{G}$. We end this section by considering the cohomology of $Leib_\infty$-conformal algebras.

\begin{definition}
    A {\em $Leib_\infty$-algebra} is a pair $(\mathcal{G} = \oplus_{i \in \mathbb{Z}} \mathcal{G}_i , \{ \pi_k \}_{k \geq 1})$ consisting of a graded vector space $\mathcal{G} = \oplus_{i \in \mathbb{Z}} \mathcal{G}_i $ equipped with a sequence of graded linear maps $\{ \pi_k : \mathcal{G}^{\otimes k} \rightarrow \mathcal{G} \}_{k \geq 1}$ with $\mathrm{deg} (\pi_k) = k-2$ for $k \geq 1$, such that for any $n \geq 1$,
    \begin{align*}
        \sum_{k+l = n+1} \sum_{i=1}^k & \sum_{\sigma \in \mathrm{Sh} (i-1, l-1)}  \varepsilon (\sigma)~ \mathrm{sgn} (\sigma)~ (-1)^{ (k-i-1)(l-1)} (-1)^{ l ( |x_{\sigma (1)}| + \cdots + |x_{\sigma (i-1)}| ) } \\
       & \pi_k \big(   x_{\sigma (1)}, \ldots, x_{\sigma (i-1)}, \pi_l \big(  x_{\sigma (i)}, \ldots, x_{\sigma (i + l-2)}, x_{i+l-1} \big), x_{i+l}, \ldots, x_n  \big) = 0.
    \end{align*}
\end{definition}

The concept of $Leib_\infty$-algebras is the homotopification of Leibniz algebras. More precisely, a $Leib_\infty$-algebra whose underlying graded vector space is concentrated only in degree $0$ is nothing but a Leibniz algebra. In \cite{ammar-poncin}, the authors showed that $Leib_\infty$-algebras can be better understood if we shift the degree of the underlying graded vector space. Using such observation, one can also give a Maurer-Cartan characterization of $Leib_\infty$-algebras (see also \cite{xu,uchino}).

Let $\mathcal{G} = \oplus_{i \in \mathbb{Z}} \mathcal{G}_i$ be a graded $\mathbb{C}[\partial]$-module. Then for any indeterminates $\lambda_1, \ldots, \lambda_{k-1}$ (for $k \geq 1$), the space $\mathcal{G} [[  \lambda_1, \ldots, \lambda_{k-1} ]]$ is a graded vector space over $\mathbb{C}$ with the grading 
\begin{align*}
\mathcal{G} [[  \lambda_1, \ldots, \lambda_{k-1} ]] = \oplus_{i \in \mathbb{Z}} \mathcal{G}_i [[ \lambda_1, \ldots, \lambda_{k-1}  ]].
\end{align*}

\begin{definition}
    A {\em $Leib_\infty$-conformal algebra} (also called a {\em strongly homotopy Leibniz conformal algebra}) is a pair $(\mathcal{G} = \oplus_{i \in \mathbb{Z}} \mathcal{G}_i , \{ \rho_k \}_{k \geq 1})$ consisting of a graded $\mathbb{C} [\partial]$-module $\mathcal{G} = \oplus_{i \in \mathbb{Z}} \mathcal{G}_i$ with a collection
    \begin{align*}
        \{   \rho_k : \mathcal{G}^{\otimes k} \rightarrow \mathcal{G} [[  \lambda_1, \ldots, \lambda_{k-1}  ]], ~ x_1 \otimes \cdots \otimes x_k \mapsto (\rho_k)_{\lambda_1, \ldots, \lambda_{k-1}} (x_1, \ldots, x_k)  \}_{k \geq 1}
    \end{align*}
    of graded $\mathbb{C}$-linear maps with $\mathrm{deg} (\rho_k) = k-2$ for $k \geq 1$, satisfying the following conditions:

    - each $\rho_k$ is conformal sesquilinear, i.e.
    \begin{align*}
        (\rho_k)_{\lambda_1, \ldots. \lambda_{k-1}} (x_1, \ldots, \partial x_i, \ldots, x_{k-1}, x_k) =~& - \lambda_i (\rho_k)_{\lambda_1, \ldots. \lambda_{k-1}} (x_1, \ldots,  x_k), \text{ for } 1 \leq i \leq k-1, \\
        (\rho_k)_{\lambda_1, \ldots. \lambda_{k-1}} (x_1, \ldots, x_{k-1}, \partial x_k) =~& (\partial + \lambda_1 + \cdots + \lambda_{k-1}) ~(\rho_k)_{\lambda_1, \ldots. \lambda_{k-1}} (x_1, \ldots, \partial x_i, \ldots,  x_k),
    \end{align*}

    - for each $n \geq 1$ and homogeneous elements $x_1, \ldots, x_n \in \mathcal{G}$,
     \begin{align}\label{iden-lca}
        \sum_{k+l = n+1} &\sum_{i=1}^k  \sum_{\sigma \in \mathrm{Sh} (i-1, l-1)}  \varepsilon (\sigma)~ \mathrm{sgn} (\sigma)~ (-1)^{ (k-i-1)(l-1)} (-1)^{ l ( |x_{\sigma (1)}| + \cdots + |x_{\sigma (i-1)}| ) } \\
       & (\rho_k)_{ \lambda_{\sigma (1)}, \ldots, \lambda_{\sigma (i-1)} ,  \lambda_{\sigma (i)} + \cdots + \lambda_{\sigma (i+l-2)} + \lambda_{i+l-1} , \ldots, \lambda_{n-1}  } \big(   x_{\sigma (1)}, \ldots, x_{\sigma (i-1)}, \nonumber \\
      & \qquad \qquad \qquad \qquad \qquad (\rho_l)_{  \lambda_{\sigma (i)} , \ldots , \lambda_{\sigma (i+l-2)} } \big(  x_{\sigma (i)}, \ldots, x_{\sigma (i + l-2)}, x_{i+l-1} \big), x_{i+l}, \ldots, x_n  \big) = 0. \nonumber
       \end{align}
\end{definition} 

\medskip

The above identities (\ref{iden-lca}), called conformal Leibnizator identities, have meaningful interpretations in low values of $n$. For instance, when $n=1$, it says that the degree $-1$ $\mathbb{C} [\partial]$-linear map $\rho_1 : \mathcal{G} \rightarrow \mathcal{G}$ satisfies $(\rho_1)^2 = 0$. In other words, $(\mathcal{G} = \oplus_{i \in \mathbb{Z}} \mathcal{G}_i , \rho_1)$ is a chain complex in the category of $\mathbb{C} [\partial]$-modules. For $n =2$, it says that the degree $0$ conformal sesquilinear map (which we call the $\lambda$-bracket in this case) $\rho_2 : \mathcal{G} \otimes \mathcal{G} \rightarrow \mathcal{G} [[\lambda ]]$, $x \otimes y \mapsto (\rho_2)_\lambda (x, y) = [x_\lambda y]$ satisfies
\begin{align*}
    \rho_1 [x_\lambda y] = [\rho_1 (x)_\lambda y] + (-1)^{|x|} [x_\lambda \rho_1 (y)],
\end{align*}
for all homogeneous elements $x, y \in \mathcal{G}$. Similarly, for $n=3$, the identity (\ref{iden-lca}) simply means
\begin{align*}
  [x_\lambda [y_\mu z]] - & [[x_\lambda y]_{\lambda + \mu} z] - (-1)^{|x| |y|} [y_\mu [x_\lambda z]] 
  = (\rho_1) \big(  (\rho_3)_{\lambda , \mu} (x, y, z) \big) \\
  &+ (\rho_3)_{\lambda, \mu} \big(  \rho_1 (x), y, z  \big) + (-1)^{|x|} (\rho_3)_{\lambda, \mu} \big(  x, \rho_1 ( y) , z  \big)  + (-1)^{|x| + |y|} (\rho_3)_{\lambda, \mu} \big( x, y, \rho_1 (z)  \big),
\end{align*}
 for $ x, y, z \in \mathcal{G}.$ This shows that the $\lambda$-bracket $[\cdot_\lambda \cdot ]$ satisfies the graded Leibniz conformal identity up to an exact term of $\rho_3$ (i.e. up to homotopy). Similarly, we obtain higher identities for higher values of $n$.

\begin{exam}
    Any Leibniz conformal algebra $(\mathfrak{g}, [\cdot_\lambda \cdot])$ is a $Leib_\infty$-conformal algebra whose underlying graded $\mathbb{C}[\partial]$-module is $\mathfrak{g}$ concentrated in degree $0$.
\end{exam}

\begin{exam}
    Any graded Leibniz conformal algebra $(\mathcal{G} = \oplus_{i \in \mathbb{Z}} \mathcal{G}_i , [\cdot_\lambda \cdot])$ is a $Leib_\infty$-conformal algebra, where $(\rho_2)_\lambda = [\cdot_\lambda \cdot]$ and $\rho_k = 0$ for $k \neq 2$. Similarly, a differential graded Leibniz conformal algebra $(\mathcal{G}, [\cdot_\lambda \cdot], d)$ is a $Leib_\infty$-conformal algebra, where
    \begin{align*}
        \rho_1 = d, ~~~~ (\rho_2)_\lambda = [\cdot_\lambda \cdot] ~~ \text{ and } ~~ \rho_k = 0 ~ \text{ for } k \geq 3.
    \end{align*}
\end{exam}

\begin{exam}
 Let $\mathfrak{g}, \mathfrak{h}$ be two Leibniz conformal algebras and $f : \mathfrak{g} \rightarrow \mathfrak{h}$ be a morphism of Leibniz conformal algebras (i.e. $f: \mathfrak{g} \rightarrow \mathfrak{h}$ is a $\mathbb{C} [\partial]$-linear map satisfying $f ([x_\lambda y]^\mathfrak{g}) = [  f(x)_\lambda f(y)]^\mathfrak{h}$, for all $x, y \in \mathfrak{g}$). Then it follows that $\mathrm{ker }f$ is a $\mathbb{C} [\partial]$-module. Moreover, the graded $\mathbb{C} [\partial]$-module $\mathcal{G} = \mathfrak{g} \oplus \mathrm{ker } f$ (where $\mathfrak{g}$ is concentrated in arity $0$ and $\mathrm{ker } f$ is concentrated in arity $1$) inherits a $Leib_\infty$-conformal algebra structure with the operations
 \begin{align*}
     \rho_1 := i : \mathrm{ker }f \hookrightarrow \mathfrak{g}, ~~~~ (\rho_2)_\lambda (x, y) := [x_\lambda y]^\mathfrak{g} \text{ for } x, y \in \mathfrak{g} \text{ or in } \mathrm{ker }f, \text{ and } \rho_k = 0 \text{ for } k \geq 3.
 \end{align*}
 \end{exam}

In \cite{sahoo-das} the authors introduced the notion of $L_\infty$-conformal algebras that are the conformal analogues of $L_\infty$-algebras and homotopification of Lie conformal algebras. It turns out that a $Leib_\infty$-conformal algebra $(\mathcal{G} = \oplus_{i \in \mathbb{Z}} \mathcal{G}_i , \{ \rho_k \}_{k \geq 1})$ in which

(i) the structure maps $\mu_k$'s are graded skew-symmetric in the sense that
\begin{align*}
    (\rho_k)_{\lambda_1, \ldots, \lambda_{k-1}} (x_1, \ldots, x_k ) = \varepsilon (\sigma)~ \mathrm{sgn}(\sigma) (\rho_k)_{\lambda_{\sigma (1)}, \ldots, \lambda_{\sigma (k-1)}} \big( x_{\sigma (1)}, \ldots, x_{\sigma (k-1)}, x_{\sigma (k)} \big) \bigg|_{\lambda_k \mapsto \lambda_k^\dagger}, \text{ for any } \sigma \in S_k
\end{align*}
(here the notation $\lambda_k \mapsto \lambda_k^\dagger$ means that $\lambda_k$ is replaced by $\lambda_k^\dagger = - \partial - \sum_{i=1}^{k-1} \lambda_i$, if it occurs and $\partial$ moved to the left),

(ii) for each $k \geq 1$, the image of the map $\rho_k$ lies in the space of $\mathfrak{g}$-valued polynomials in $k-1$ indeterminates $\lambda_1, \ldots, \lambda_{k-1}$ (i.e. $\mathrm{Im} (\rho_k) \subset \mathcal{G} [  \lambda_1, \ldots, \lambda_{k-1} ]$ for all $k \geq 1$),\\
is nothing but a $L_\infty$-conformal algebra. Thus, a $Leib_\infty$-conformal algebra is more general than an $L_\infty$-conformal algebra where the graded skew-symmetry condition and the polynomial condition of the structure maps are relaxed. 

\medskip

It is important to remark that by removing all $\lambda_i$'s and $\partial$'s from the definition of a $Leib_\infty$-conformal algebra, one simply gets a $Leib_\infty$-algebra. Hence $Leib_\infty$-conformal algebras can be realized as the conformal analogue of $Leib_\infty$-algebras. In \cite{ammar-poncin} (see also \cite{xu,uchino}) the authors gave a simple description of a $Leib_\infty$-algebra. Namely, they first considered the notion of $Leib_\infty [1]$-algebras (which are much more convenient than $Leib_\infty$-algebras) and showed that they are equivalent to $Leib_\infty$-algebras when we shift the underlying graded space. In the following, we will adapt their approach in the context of conformal algebras. We start with the following definition.

\begin{definition}
    A {\em $Leib_\infty [1]$-conformal algebra} is a pair $( \mathcal{H} = \oplus_{i \in \mathbb{Z}} \mathcal{H}_i, \{ \varrho_k \}_{k \geq 1})$ of a graded $\mathbb{C} [\partial]$-module  $\mathcal{H} = \oplus_{i \in \mathbb{Z}} \mathcal{H}_i$ with a collection of degree $-1$ graded $\mathbb{C}$-linear maps
    \begin{align*}
        \{ \varrho_k : \mathcal{H}^{\otimes k} \rightarrow \mathcal{H} [[ \lambda_1, \ldots, \lambda_{k-1}]],~ h_1 \otimes \cdots \otimes h_k \mapsto (\varrho_k)_{\lambda_1, \ldots, \lambda_{k-1}} (h_1, \ldots,  h_k ) \}_{k \geq 1}
    \end{align*}
    that satisfy the following conditions:

    - each $\varrho_k$ is conformal sesquilinear,

    - for any $n \geq 1$ and homogeneous elements $h_1, \ldots, h_n \in \mathcal{H}$,
      \begin{align}\label{iden-lca1}
        \sum_{k+l = n+1} &\sum_{i=1}^k  \sum_{\sigma \in \mathrm{Sh} (i-1, l-1)}  \varepsilon (\sigma)~ (-1)^{  |h_{\sigma (1)}| + \cdots + |h_{\sigma (i-1)}| } \\
        &(\varrho_k)_{ \lambda_{\sigma (1)}, \ldots, \lambda_{\sigma (i-1)} ,  \lambda_{\sigma (i)} + \cdots + \lambda_{\sigma (i+l-2)} + \lambda_{i+l-1} , \ldots, \lambda_{n-1}  } \big(   h_{\sigma (1)}, \ldots, h_{\sigma (i-1)}, \nonumber \\
        & \qquad \qquad \qquad \qquad \qquad (\varrho_l)_{  \lambda_{\sigma (i)} , \ldots , \lambda_{\sigma (i+l-2)} } \big(  h_{\sigma (i)}, \ldots, h_{\sigma (i + l-2)}, h_{i+l-1} \big), h_{i+l}, \ldots, h_n  \big) = 0. \nonumber
       \end{align}
\end{definition}

\medskip

If $(\mathcal{G} = \oplus_{i \in \mathbb{Z}} \mathcal{G}_i, [ \cdot_\lambda \cdot ], d)$ is a differential graded Leibniz conformal algebra, then the shifted graded $\mathbb{C} [\partial]$-module $\mathcal{G}[-1] = \oplus_{i \in \mathbb{Z}} (\mathcal{G} [-1])_i = \oplus_{i \in \mathbb{Z}} \mathcal{G}_{i-1}$ inherits a $Leib_\infty [1]$-conformal algebra with the operations $ \{ \varrho_k: \mathcal{G}[-1]^{\otimes k} \rightarrow \mathcal{G}[-1] [[ \lambda_1, \ldots, \lambda_{k-1}]] \}_{k \geq 1}$ given by
\begin{align*}
    \varrho_1 (x) := (s \circ d) (s^{-1} (x)), ~~~~ (\varrho_2)_\lambda (x, y) := (-1)^{|x|} s [ (s^{-1} x)_\lambda (s^{-1} y ) ] ~~~ \text{ and } ~~~ \varrho_k = 0 \text{ for } k \geq 3. 
\end{align*}
Here $s : \mathcal{G} \rightarrow \mathcal{G}[-1]$ is the degree $+1$ map that identifies $\mathcal{G}$ with the shifted space $\mathcal{G}[-1]$, and $s^{-1} : \mathcal{G}[-1] \rightarrow \mathcal{G}$ is the degree $-1$ map inverses to $s$. This is not a surprising result as we have the following general result whose proof is similar to the classical case \cite{ammar-poncin}.

\begin{thm}\label{thm-lca-lca1}
    Let $\mathcal{G} = \oplus_{i \in \mathbb{Z}} \mathcal{G}_i$ be a graded $\mathbb{C}[\partial]$-module. Then there is a one-one correspondence between $Leib_\infty$-conformal algebra structures on $\mathcal{G}$ and $Leib_\infty[1]$-conformal algebra structures on $\mathcal{H} = \mathcal{G}[-1]$.
\end{thm}

Let $\mathcal{H} = \oplus_{i \in \mathbb{Z}} \mathcal{H}_i$ be a graded $\mathbb{C}[\partial]$-module. For any integer $n \in \mathbb{Z}$ and natural number $k \in \mathbb{N}$, let $\mathrm{Hom}_{cs}^n \big( \mathcal{H}^{\otimes k}, \mathcal{H} [[ \lambda_1, \ldots, \lambda_{k-1}]] \big)$ be the space of all conformal sesquilinear maps $\varphi_k : \mathcal{H}^{\otimes k} \rightarrow \mathcal{H} [[ \lambda_1, \ldots, \lambda_{k-1}]]$ with $\mathrm{deg}(\varphi_k) = n$. We define
\begin{align*}
    C^n_{cs} (\mathcal{H}) = \oplus_{k \geq 1} \mathrm{Hom}_{cs}^n \big( \mathcal{H}^{\otimes k}, \mathcal{H} [[ \lambda_1, \ldots, \lambda_{k-1}]] \big).
\end{align*}
Thus, an element $\varphi \in C^N_{cs} (\mathcal{H})$ is given by a sum $\varphi = \sum_{k \geq 1} \varphi_k$, where $\varphi_k \in \mathrm{Hom}_{cs}^n \big( \mathcal{H}^{\otimes k}, \mathcal{H} [[ \lambda_1, \ldots, \lambda_{k-1}]] \big)$ for $k \geq 1$. Note that the graded space $C^\bullet_{cs} (\mathcal{H}) = \oplus_{n \in \mathbb{Z}} C^n_{cs} (\mathcal{H})$ inherits a graded Lie bracket given by
\begin{align*}
    \llbracket \sum_{k \geq 1} \varphi_k , \sum_{l \geq 1} \psi_l \rrbracket = \sum_{p \geq 1} \sum_{k+l = p+1} (\varphi_k \diamond \psi_l - (-1)^{mn} \psi_l \diamond \varphi_k), \text{ where }
\end{align*}
\begin{align*}
   & (\varphi_k \diamond \psi_l)_{\lambda_1, \ldots, \lambda_{p-1}} (h_1, \ldots, h_p) \\
   & = \sum_{i=1}^k \sum_{\sigma \in \mathrm{Sh} (i-1, l-1)} \varepsilon (\sigma) (-1)^{m ( |h_{\sigma (1)}| + \cdots + |h_{\sigma (i-1)}|)} \\
   & \qquad \qquad \quad \qquad  (\varphi_k)_{ \lambda_{\sigma (1)}, \ldots, \lambda_{\sigma (i-1)} ,  \lambda_{\sigma (i)} + \cdots + \lambda_{\sigma (i+l-2)} + \lambda_{i+l-1} , \ldots, \lambda_{n-1}  } \big(   h_{\sigma (1)}, \ldots, h_{\sigma (i-1)}, \\ 
   & \qquad \qquad \qquad \qquad \qquad \qquad \qquad (\psi_l)_{  \lambda_{\sigma (i)} , \ldots , \lambda_{\sigma (i+l-2)} } \big(  h_{\sigma (i)}, \ldots, h_{\sigma (i + l-2)}, h_{i+l-1} \big),  \ldots, h_n  \big) = 0,
\end{align*}
for $\varphi  = \sum_{k \geq 1} \varphi_k \in C^n_{cs} (\mathcal{H})$ and $\psi = \sum_{l \geq 1} \psi_l \in C^m_{cs} (\mathcal{H})$. Therefore, $(C^\bullet_{cs} (\mathcal{H}), \llbracket ~, ~ \rrbracket)$ is a graded Lie algebra. This graded Lie algebra is the conformal analogue of the graded Lie algebra considered by Balavoive \cite{bala} that characterizes Leibniz algebras as Maurer-Cartan elements. In the present context, we have the following result.

\begin{thm}
Let $\mathcal{H} = \oplus_{i \in \mathbb{Z}} \mathcal{H}_i$ be a graded $\mathbb{C}[\partial]$-module and $\{ \varrho_k : \mathcal{H}^{\otimes k} \rightarrow \mathcal{H} [[ \lambda_1, \ldots, \lambda_{k-1} ]] \}_{k \geq 1}$ be a collection of degree $-1$ conformal sesquilinear maps. Then the pair $(\mathcal{H}, \{ \varrho_k \}_{k \geq 1})$ is a $Leib_\infty [1]$-conformal algebra if and only if the element $\varrho = \sum_{k \geq 1} \varrho_k \in C^{-1}_{cs} (\mathcal{H})$ is a Maurer-Cartan element in the graded Lie algebra $(C^\bullet_{cs} (\mathcal{H}), \llbracket ~, ~ \rrbracket)$.
\end{thm}

\begin{proof}
    Note that
    \begin{align*}
        \llbracket \varrho, \varrho \rrbracket = \llbracket \sum_{k \geq 1} \varrho_k, \sum_{l \geq 1} \varrho_l \rrbracket =~& \sum_{n \geq 1} \sum_{k+l = n+1} \sum_{k+l = n+1 } (\varrho_k \diamond \varrho_l - (-1)^1 \varrho_l \diamond \varrho_k) \\
        =~& 2 \sum_{n \geq 1} \sum_{k+l = n+1} \varrho_k \diamond \varrho_l.
    \end{align*}
    This shows that $\varrho$ is a Maurer-Cartan element if and only if $\sum_{k+l = n+1} \varrho_k \diamond \varrho_l = 0$ for all $n \geq 1$. This is equivalent to the fact that $(\mathcal{H} , \{ \varrho_k \}_{k \geq 1})$ defines a $Leib_\infty [1]$-conformal algebra structure on $\mathcal{H}$.
\end{proof}

The above theorem gives rise to a characterization of a $Leib_\infty [1]$-conformal algebra in terms of a Maurer-Cartan element in a graded Lie algebra. On the other hand, in Theorem \ref{thm-lca-lca1}, we have already shown that $Leib_\infty$-conformal algebras and $Leib_\infty [1]$-conformal algebras are equivalent up to a degree shift. Combining these results, we get the following description of a $Leib_\infty$-conformal algebra.

\begin{thm}
    Let $\mathcal{G}$ be a graded $\mathbb{C}[\partial]$-module. Then a $Leib_\infty$-conformal algebra structure on $\mathcal{G}$ is equivalent to a Maurer-Cartan element in the graded Lie algebra $(C^\bullet_{cs} (\mathcal{G}[-1]), \llbracket ~, ~ \rrbracket ).$
\end{thm}

In the following, we will focus on the cohomology of a $Leib_\infty$-conformal algebra using the above Maurer-Cartan characterization. Let $(\mathcal{G}, \{ \rho_k \}_{k \geq 1})$ be a $Leib_\infty$-conformal algebra. Consider the corresponding Maurer-Cartan element $\varrho = \sum_{k \geq 1} \varrho_k$, where $\varrho_k = (-1)^{\frac{k(k-1)}{2}} s \circ \rho_k \circ (s^{-1})^{\otimes k}$ for any $k \geq 1$.

For each $n \in \mathbb{Z}$, we define $C^n_{\mathrm{cLeib_\infty}} (\mathcal{G}, \mathcal{G}) := C^{- (n-1)}_{cs} (\mathcal{G}[-1])$. It follows that an element $\varphi \in C^n_{\mathrm{cLeib_\infty}} (\mathcal{G}, \mathcal{G})$ is a sum $\varphi = \sum_{k \geq 1} \varphi_k$, where $\varphi_k \in \mathrm{Hom}_{cs}^{- (n-1)}  \big(  \mathcal{G}[-1]^{\otimes k}, \mathcal{G}[-1] [[ \lambda_1, \ldots, \lambda_{k-1}]]  \big)$ for $k \geq 1$. We also define a map $\delta : C^n_{\mathrm{cLeib_\infty}} (\mathcal{G}, \mathcal{G}) \rightarrow C^{n+1}_{\mathrm{cLeib_\infty}} (\mathcal{G}, \mathcal{G})$ by
\begin{align}\label{cob}
    \delta (\varphi) = (-1)^{n-1} \llbracket \varrho, \varphi \rrbracket, \text{ for } \varphi = \sum_{k \geq 1} \varphi_k \in C^n_{\mathrm{cLeib_\infty}} (\mathcal{G}, \mathcal{G}) = C_{cs}^{- (n-1)} (\mathcal{G}[-1]).
\end{align}
Then we have the following.

\begin{proposition}
    Let $(\mathcal{G}, \{ \rho_k \}_{k \geq 1})$ be a $Leib_\infty$-conformal algebra. Then $(C^\bullet_{\mathrm{cLeib_\infty}} (\mathcal{G}, \mathcal{G}), \delta)$ is a cochain complex.
\end{proposition}

The proof of the above proposition follows as $\varrho = \sum_{k \geq } \varrho_k$ is a Maurer-Cartan element in the graded Lie algebra $(C^\bullet_{cs} (\mathcal{G}[-1]),   \llbracket ~, ~ \rrbracket)$. The cohomology groups of the cochain complex $(C^\bullet_{\mathrm{cLeib_\infty}} (\mathcal{G}, \mathcal{G}), \delta)$ are called the {\em cohomology} of the given $Leib_\infty$-conformal algebra $\mathcal{G}$.

\begin{remark}
The above definition of cohomology can be easily generalized in the presence of a representation. Let $\mathcal{G} = ( \mathcal{G} = \oplus_{i \in \mathbb{Z}} \mathcal{G}_i, \{ \rho_k \}_{k \geq 1})$ be a $Leib_\infty$-conformal algebra. A {\em representation} of $\mathcal{G}$ is a graded $\mathbb{C}[\partial]$-module $\mathcal{M} = \oplus_{i \in \mathbb{Z}} \mathcal{M}_i$ equipped with a collection of $\mathbb{C}$-linear conformal sesquilinear maps
\begin{align*}
    \big\{   \theta_k : \oplus_{i=0}^k (\mathcal{G}^{\otimes (i-1)} \otimes \mathcal{M} \otimes \mathcal{G}^{\otimes (k-i)}) \rightarrow \mathcal{M} [[\lambda_1, \ldots, \lambda_{k-1}]]  \big\}_{k \geq 1}
\end{align*}
with $\mathrm{deg} (\theta_k) = k-2$ for $k \geq 1$, such that the conformal Leibnizator identities (\ref{iden-lca}) are hold when exactly one of the inputs among $x_1, \ldots, x_n$ comes from $\mathcal{M}$ and the corresponding $\rho_k$ or $\rho_l$ are replaced by $\theta_k$ or $\theta_l$. It turns out that any $Leib_\infty$-conformal algebra is a representation of itself, where $\theta_k = \rho_k$ for all $k \geq 1$. We call this the adjoint representation.

Given a $Leib_\infty$-conformal algebra $(\mathcal{G}, \{ \rho_k \}_{k \geq 1})$ and a representation $\mathcal{M}$, we now define a cochain complex $\{ C^\bullet_{\mathrm{cLeib_\infty}} (\mathcal{G}, \mathcal{M}), \delta \}$ as follows: for each $n \in \mathbb{Z}$, 
\begin{align*}
    C^n_{\mathrm{cLeib_\infty}}  (\mathcal{G}, \mathcal{M}) = \oplus_{k \geq 1} \mathrm{Hom}_{cs}^n \big(  \mathcal{G} [-1]^{\otimes k}, \mathcal{M}[-1] [[ \lambda_1, \ldots, \lambda_{k-1}]] \big)
\end{align*}
and the coboundary map $\delta :   C^n_{\mathrm{cLeib_\infty}}  (\mathcal{G}, \mathcal{M}) \rightarrow C^{n+1}_{\mathrm{cLeib_\infty}} (\mathcal{G}, \mathcal{M})$ can be defined similar to (\ref{cob}). The cohomology groups of the complex $\{ C^\bullet_{\mathrm{cLeib_\infty}} (\mathcal{G}, \mathcal{M}), \delta \}$ are called the {\em cohomology} of $\mathcal{G}$ with coefficients in the representation $\mathcal{M}$. Note that the cohomology of $\mathcal{G}$ defined earlier is nothing but the cohomology with coefficients in the adjoint representation.
\end{remark}

\section{Skeletal and strict homotopy Leibniz conformal algebras}\label{sec4}
In this section, we consider some particular classes of $Leib_\infty$-conformal algebras. Specifically, we study `skeletal' and `strict' $Leib_\infty$-conformal algebras. Among others, we characterize skeletal algebras by third cocycles of Leibniz conformal algebras and strict algebras by crossed modules of Leibniz conformal algebras.

Our main motivations for the results come from a novel work of Baez and Crans \cite{baez-crans}. They mainly considered skeletal and strict $L_\infty$-algebras and characterized them in terms of suitable objects related to Lie algebras. In other words, they relate some particular classes of $L_\infty$-algberas and invariants associated to Lie algebras. We mainly generalize their results in the context of Leibniz conformal algebras. We first start with the following.

\begin{definition}\label{2termdef}
    A {\bf $2$-term $Leib_\infty$-conformal algebra} is a triple $(\mathcal{G}_1 \xrightarrow{d} \mathcal{G}_0, \rho_2, \rho_3)$ consisting of a complex $\mathcal{G}_1 \xrightarrow{d} \mathcal{G}_0$ of $\mathbb{C}[\partial]$-modules equipped with

    - a $\mathbb{C}$-linear conformal sesquilinear map $\rho_2 : \mathcal{G}_i \otimes \mathcal{G}_j \rightarrow \mathcal{G}_{i+j} [[ \lambda ]]$, for $0 \leq i, j , i+j \leq 1$, 

    - a $\mathbb{C}$-linear conformal sesquilinear map $\rho_3 : \mathcal{G}_0 \otimes \mathcal{G}_0 \otimes \mathcal{G}_0 \rightarrow \mathcal{G}_1 [[ \lambda, \mu ]]$\\
    that satisfy the following set of identities: for all $x, y,z, w \in \mathcal{G}_0$ and $u, v \in \mathcal{G}_1$,

    \medskip

    \begin{itemize}
        \item[(i)] $(\rho_2)_\lambda (u, v) = 0$,

        \medskip
        
        \item[(ii)] $d (  (\rho_2)_\lambda (x, u) ) = (\rho_2)_\lambda (x, du),$
        
        \medskip

        \item[(iii)] $d (  (\rho_2)_\lambda (u, x) ) = (\rho_2)_\lambda ( du, x),$

        \medskip

        \item[(iv)] $(\rho_2)_\lambda (du, v) = (\rho_2)_\lambda (u, dv),$ 

        \medskip
        
        \item[(v)] $d \big(  (\rho_3)_{\lambda, \mu} (x, y, z)  \big) = (\rho_2)_\lambda  (x,  (\rho_2)_\mu (y, z) ) -  (\rho_2)_{\lambda + \mu}  (   (\rho_2)_\lambda (x, y), z ) - (\rho_2)_\mu (y, (\rho_2)_\lambda (x,z)),$

        \medskip
        
        \item[(vi)] $  (\rho_3)_{\lambda, \mu} (x, y, dv)  = (\rho_2)_\lambda  (x,  (\rho_2)_\mu (y, v) ) -  (\rho_2)_{\lambda + \mu}  (   (\rho_2)_\lambda (x, y), v ) - (\rho_2)_\mu (y, (\rho_2)_\lambda (x,v)),$

        \medskip
        
        \item[(vii)] $  (\rho_3)_{\lambda, \mu} (x, dv, y)   = (\rho_2)_\lambda  (x,  (\rho_2)_\mu (v, y) ) -  (\rho_2)_{\lambda + \mu}  (   (\rho_2)_\lambda (x, v), y) - (\rho_2)_\mu (v, (\rho_2)_\lambda (x,y)),$

        \medskip
        
        \item[(viii)] $  (\rho_3)_{\lambda, \mu} (dv, x, y)  = (\rho_2)_\lambda  (v,  (\rho_2)_\mu (x, y) ) -  (\rho_2)_{\lambda + \mu}  (   (\rho_2)_\lambda (v, x), y ) - (\rho_2)_\mu (x, (\rho_2)_\lambda (v, y)),$

        \medskip
        
        \item[(ix)] the conformal Leibnizator identity:



        

         
    \end{itemize}
    \begin{align*}
       &  (\rho_2)_\lambda \big(  x, (\rho_3)_{\mu, \nu} (y, z, w)   \big) - (\rho_2)_\mu \big(  y, (\rho_3)_{\lambda, \nu} (x, z, w) \big) + (\rho_2)_\nu \big( z, (\rho_3)_{\lambda, \mu} (x, y, w) \big) \\
       + (\rho_2&)_{\lambda + \mu + \nu}\big( (\rho_3)_{\lambda, \mu} (x, y, z ), w \big) 
        - (\rho_3)_{\lambda + \mu, \nu} \big(  (\rho_2)_\lambda (x, y), z, w  \big) - (\rho_3)_{\mu, \lambda + \nu} \big( y, (\rho_2)_{\lambda} (x, z), w \big) \\ 
        &- (\rho_3)_{\mu, \nu} \big( y, z, (\rho_2)_{\lambda} (x, w)  \big)
        + (\rho_3)_{\lambda, \mu+ \nu} \big(  x, (\rho_2)_{\mu} (y, z), w \big) + (\rho_3)_{\lambda, \nu} \big(x, z,     (\rho_2)_{\mu} (y, w) \big) \\
        & \qquad \qquad \qquad \qquad \qquad   - (\rho_3)_{\lambda, \mu} \big( x, y, (\rho_2)_{\nu} (z, w) \big) = 0.
    \end{align*}
\end{definition}

\medskip

It follows from the above definition that a $2$-term $Leib_\infty$-conformal algebra $(\mathcal{G}_1 \xrightarrow{d} \mathcal{G}_0, \rho_2, \rho_3)$ is nothing but a $Leib_\infty$-conformal algebra $(\mathcal{G} = \mathcal{G}_0 \oplus \mathcal{G}_1, \{ \rho_k \}_{k \geq 1})$ with $\rho_1 = d$ and $\rho_k = 0$ for $k \geq 4$. In other words, a $2$-term $Leib_\infty$-conformal algebra is a $Leib_\infty$-conformal algebra whose underlying graded $\mathbb{C} [\partial]$-module is concentrated in degree $0$ and degree $1$.

Let $\mathcal{G} = (\mathcal{G}_1 \xrightarrow{d} \mathcal{G}_0, \rho_2, \rho_3)$ and $\mathcal{G}' = (\mathcal{G}'_1 \xrightarrow{d'} \mathcal{G}'_0, \rho'_2, \rho'_3)$ be two $2$-term $Leib_\infty$-conformal algebras. A {\em homomorphism} from $\mathcal{G}$ to $\mathcal{G}'$ is given by a triple $f = (f_0, f_1, f_2)$ in which $f_0 : \mathcal{G}_0 \rightarrow \mathcal{G}_0'$ and $f_1 : \mathcal{G}_1 \rightarrow \mathcal{G}_1'$
 are $\mathbb{C} [\partial]$-linear maps and $f_2 : \mathcal{G}_0 \otimes \mathcal{G}_0 \rightarrow \mathcal{G}_1' [[ \lambda ]]$ is a $\mathbb{C}$-linear conformal sesquilinear map satisfying the following conditions: for any $x, y, z \in \mathcal{G}_0$ and $v \in \mathcal{G}_1$,

 \medskip

 $\bullet ~~$ $d' \circ f_1 = f_0 \circ d,$

 \medskip
   
 $\bullet ~~$ $(\rho'_2)_\lambda \big(  f_0 (x), f_0 (y) \big) - f_0 \big(  (\rho_2)_\lambda (x, y)   \big) = d' \big(  (f_2)_\lambda (x, y)  \big),$
    
 \medskip

 $\bullet ~~$  $(\rho_2')_\lambda \big(  f_0 (x), f_1 (v)  \big) - f_1 \big(  (\rho_2)_\lambda (x, v) \big) = (f_2)_\lambda (x, dv),$
    
 \medskip

 $\bullet ~~$  $ (\rho_2')_\lambda \big(  f_1 (v) , f_0 (x)  \big) - f_1 \big(  (\rho_2)_\lambda (v, x) \big) = (f_2)_\lambda ( dv, x),$ 
    
 \medskip

 $\bullet ~~$  $\resizebox{0.4\textwidth}{!}{$(\rho'_3)_{\lambda, \mu} \big(  f_0 (x), f_0 (y), f_0 (z)  \big) - f_1 \big(  (\rho_3)_{\lambda, \mu} (x, y, z) \big)$} = (\rho'_2)_\lambda \big(  f_0 (x), (f_2)_\mu (y, z) \big) - (\rho'_2)_{\lambda + \mu} \big(  (f_2)_\lambda (x, y), f_0 (z) \big)$

\medskip
 
 $ \quad - (\rho'_2)_\mu \big(  f_0 (y), (f_2)_\lambda (x, z)  \big) + (f_2)_\lambda \big(  x, (\rho_2)_\mu (y, z) \big) - (f_2)_{\lambda + \mu} \big(  (\rho_2)_\lambda (x, y), z \big) - (f_2)_\mu \big( y, (\rho_2)_\lambda (x, z)  \big).$

\medskip

\noindent If $\mathcal{G} = (\mathcal{G}_1 \xrightarrow{d} \mathcal{G}_0, \rho_2, \rho_3)$ is a $2$-term $Leib_\infty$-conformal algebra, then the identity homomorphism from $\mathcal{G}$ to $\mathcal{G}$ is given by the triple $\mathrm{id}_\mathcal{G} = \big( \mathrm{id}_{\mathcal{G}_0},   \mathrm{id}_{\mathcal{G}_1}, 0 \big)$. Next, if $f = (f_0, f_1, f_2)$ is a homomorphism from $\mathcal{G}$ to $\mathcal{G}'$ and $g = (g_0, g_1, g_2)$ is a homomorphism from $\mathcal{G}'$ to $\mathcal{G}''$, then their composition is given by the triple $g \circ f = \big(  g_0 \circ f_0 , g_1 \circ f_1, g_2 \circ (f_0 \otimes f_0) + g_1 \circ f_2  \big)$. With the above notations, one can verify the above result.

 \begin{thm}
     The collection of all $2$-term $Leib_\infty$-conformal algebras and homomorphisms between them forms a category, denoted by ${\bf 2Leib_\infty Conf}$.
 \end{thm}

 \medskip

\noindent {\bf Skeletal algebras.} Here we will focus on skeletal $Leib_\infty$-conformal algebras and their characterization in terms of third cocycles in Leibniz conformal algebras. Skeletal $Leib_\infty$-conformal algebras are a particular case of $2$-term $Leib_\infty$-conformal algebras.

\begin{definition}
    A {\em skeletal $Leib_\infty$-conformal algebra} is a $2$-term $Leib_\infty$-conformal algebra $(\mathcal{G}_1 \xrightarrow{d} \mathcal{G}_0, \rho_2, \rho_3)$ in which $d = 0$.
\end{definition}

Let $(\mathcal{G}_1 \xrightarrow{0} \mathcal{G}_0, \rho_2, \rho_3)$ be a skeletal $Leib_\infty$-conformal algebra. Then it follows from condition (v) of Definition \ref{2termdef} that the $\mathbb{C} [\partial]$-module $\mathcal{G}_0$ with the $\lambda$-bracket $[\cdot_\lambda \cdot] : \mathcal{G}_0 \otimes \mathcal{G}_0 \rightarrow \mathcal{G}_0 [[ \lambda ]]$ defined by $[x_\lambda y ]:= (\rho_2)_\lambda (x, y)$, for $x, y \in \mathcal{G}_0$, is a Leibniz conformal algebra. Similarly, the conditions (vi), (vii) and (viii) of Definition \ref{2termdef} implies that the $\mathbb{C}[\partial]$-module $\mathcal{G}_1$ can be given a representation of the Leibniz conformal algebra $(\mathcal{G}_0, [ \cdot_\lambda \cdot ])$ with the left and right $\lambda$-actions given by
\begin{align*}
    x_\lambda v := (\rho_2)_\lambda (x, v) ~~~ \text{ and } ~~~ v_\lambda x := (\rho_2)_\lambda (v, x), \text{ for } x \in \mathcal{G}_0, v \in \mathcal{G}_1.
\end{align*}
Finally, the condition (ix) of Definition \ref{2termdef} simply means that the map $\rho_3 : \mathcal{G}_0 \otimes \mathcal{G}_0 \otimes \mathcal{G}_0 \rightarrow \mathcal{G}_1 [[ \lambda, \mu ]]$ is a $3$-cocycle of the Leibniz conformal algebra $(\mathcal{G}_0, [ \cdot_\lambda \cdot])$ with coefficients in the representation $\mathcal{G}_1$.

Next, let $(\mathcal{G}_1 \xrightarrow{0} \mathcal{G}_0, \rho_2, \rho_3)$ and $(\mathcal{G}_1 \xrightarrow{0} \mathcal{G}_0, \rho'_2, \rho'_3)$ be two skeletal $Leib_\infty$-conformal algebras on the same underlying chain complex. We call them {\em equivalent} if $\rho_2 = \rho_2'$ and the corresponding $3$-cocycles $\rho_3$ and $\rho_3'$ are differ by a coboundary, i.e. there exists a conformal sesquilinear map $\tau : \mathcal{G}_0 \otimes \mathcal{G}_0 \rightarrow \mathcal{G}_1 [[ \lambda ]]$ such that $\rho_3' = \rho_3 + \delta \tau$. Here $\delta$ is the coboundary map of the Leibniz conformal algebra $\mathcal{G}_0$ with coefficients in the representation $\mathcal{G}_1$. Then we have the following result.

\begin{thm}
    There is a bijection correspondence between the collections of all skeletal $Leib_\infty$-conformal algebras and the collections of all tuples of the form $(\mathfrak{g}, M, \theta)$, where $\mathfrak{g}$ is Leibniz conformal algebra, $M$ is a representation and $\theta \in Z^3 (\mathfrak{g}, M)$ is a $3$-cocycle. Moreover, the above correspondence further extends to a bijection between the equivalence classes of all skeletal $Leib_\infty$-conformal algebras and tuples of the form $(\mathfrak{g}, M, [\theta])$, where $[\theta] \in H^3 (\mathfrak{g}, M)$ is a third cohomology class.
\end{thm}

\begin{proof}
    Let $(\mathcal{G}_1 \xrightarrow{ 0} \mathcal{G}_0, \rho_2, \rho_3)$ be a skeletal $Leib_\infty$-conformal algebra. Then we have seen earlier that $\mathcal{G}_0$ is a Leibniz conformal algebra, $\mathcal{G}_1$ is a representation and $\rho_3 \in Z^3 (\mathcal{G}_0, \mathcal{G}_1)$ is a $3$-cocycle. Hence, we obtain a required triple $(\mathcal{G}_0, \mathcal{G}_1, \rho_3)$. Conversely, let $(\mathfrak{g}, M, \theta)$ be a given triple. Then it is easy to verify that $(M \xrightarrow{ 0} \mathfrak{g}, \rho_2, \theta)$ is a skeletal $Leib_\infty$-conformal algebra, where
    \begin{align*}
        (\rho_2)_\lambda (x, y) := [x_\lambda y], ~~~ (\rho_2)_\lambda (x, v) := x_\lambda v ~~ \text{ and } ~~ (\rho_2)_\lambda (v, x) := v_\lambda x, \text{ for } x, y \in \mathfrak{g}, v \in M.
    \end{align*}

    Next, let  $(\mathcal{G}_1 \xrightarrow{ 0} \mathcal{G}_0, \rho_2, \rho_3)$ and $(\mathcal{G}_1 \xrightarrow{ 0} \mathcal{G}_0, \rho_2, \rho'_3)$ be two equivalent skeletal $Leib_\infty$-conformal algebras. Then we have $\rho'_3 = \rho_3 + \delta \tau$, for some $\mathbb{C}$-linear conformal sesquilinear map $\sigma : \mathcal{G}_0 \otimes \mathcal{G}_0 \rightarrow \mathcal{G}_1 [[ \lambda ]]$. Then it follows that $[\rho_3] = [\rho'_3]$ as an element of $H^3 (\mathcal{G}_0, \mathcal{G}_1)$. Hence we obtain a triple $(\mathcal{G}_0, \mathcal{G}_1, [\rho_3] = [\rho_3'])$. Conversely, any triple $(\mathfrak{g}, M, [\theta])$ corresponds to the equivalence class of the skeletal $Leib_\infty$-conformal algebra $(M \xrightarrow{0} \mathfrak{g}, \rho_2, \theta)$. This completes the proof.
\end{proof}

The following example is a consequence of the above result.

\begin{exam}
    Let $(\mathfrak{g}, [\cdot_\lambda \cdot])$ be a Leibniz conformal algebra and $M$ be a representation of it. For any $\mathbb{C}$-linear conformal sesquilinear map $\tau : \mathfrak{g} \otimes \mathfrak{g} \rightarrow M [[ \lambda ]]$, the triple $(M \xrightarrow{0} \mathfrak{g}, \rho_2, \delta \tau)$ is a skeletal $Leib_\infty$-conformal algebra.
\end{exam}

\medskip

\noindent {\bf Strict algebras.} Here we will consider another class of $2$-term $Leib_\infty$-conformal algebras, called strict $Leib_\infty$-conformal algebras. We also introduce crossed modules of Leibniz conformal algebras and find their relationship with strict $Leib_\infty$-conformal algebras.

\begin{definition}
    A {\em strict $Leib_\infty$-conformal algebra} is a $2$-term $Leib_\infty$-conformal algebra $(\mathcal{G}_1 \xrightarrow{ d} \mathcal{G}_0, \rho_2, \rho_3)$ in which $\rho_3 = 0$.
\end{definition}

Thus, in a strict $Leib_\infty$-conformal algebra $(\mathcal{G}_1 \xrightarrow{d} \mathcal{G}_0, \rho_2, \rho_3 = 0)$, we have from condition (v) of Definition \ref{2termdef} that
\begin{align*}
    (\rho_2)_\lambda  (x,  (\rho_2)_\mu (y, z) ) =  (\rho_2)_{\lambda + \mu}  (   (\rho_2)_\lambda (x, y), z ) + (\rho_2)_\mu (y, (\rho_2)_\lambda (x,z)), \\
    (\rho_2)_\mu  (y,  (\rho_2)_\lambda (x, z) ) =  (\rho_2)_{\lambda + \mu}  (   (\rho_2)_\mu (y, x), z ) + (\rho_2)_\lambda (x, (\rho_2)_\mu (y,z)),
\end{align*}
for all $x, y, z \in \mathcal{G}_0$. By adding both sides of the above identities, we get that
\begin{align*}
     (\rho_2)_{\lambda + \mu}  (   (\rho_2)_\lambda (x, y), z )  +  (\rho_2)_{\lambda + \mu}  (   (\rho_2)_\mu (y, x), z ) = 0, \text{ for all }x, y, z \in \mathcal{G}_0.
\end{align*}
Similarly, one can show that 
\begin{align*}
     (\rho_2)_{\lambda + \mu}  (   (\rho_2)_\lambda (x, v), y )  +  (\rho_2)_{\lambda + \mu}  (   (\rho_2)_\mu (v, x), y ) = 0, \text{ for all }x, y \in \mathcal{G}_0 \text{ and } v \in \mathcal{G}_1.
\end{align*}

\medskip

Let $(\mathcal{G}_1 \xrightarrow{ d} \mathcal{G}_0, \rho_2, 0)$ be a strict $Leib_\infty$-conformal algebra. Then it follows from condition (v) of Definition \ref{2termdef} that the $\mathbb{C} [\partial]$-module $\mathcal{G}_0$ with the $\lambda$-bracket $[ \cdot_\lambda \cdot ]^0 : \mathcal{G}_0 \otimes \mathcal{G}_0 \rightarrow \mathcal{G}_0 [[\lambda ]]$, $[x_\lambda y]^0 := (\rho_2)_\lambda (x, y)$, for $x, y \in \mathcal{G}_0$, is a Leibniz conformal algebra. We also define a conformal sesquilinear map $[ \cdot_\lambda \cdot ]^1 : \mathcal{G}_1 \otimes \mathcal{G}_1 \rightarrow \mathcal{G}[[ \lambda ]]$ by $[u_\lambda v]^1 := (\rho_2)_\lambda (du, v) = (\rho_2)_\lambda (u, dv)$, for $u, v \in \mathcal{G}_1$. Then for any $u, v, w \in \mathcal{G}_1$, we have
\begin{align*}
    [u_\lambda [v_\mu w]^1]^1 =~& [u_\lambda (\rho_2)_\mu (dv, w)]^1 \\
    =~& (\rho_2)_\lambda \big(  du,  (\rho_2)_\mu (dv, w) \big)\\
    =~& (\rho_2)_{\lambda + \mu} \big(  (\rho_2)_\lambda (du, dv), w  \big) + (\rho_2)_\mu \big(    dv, (\rho_2)_\lambda (du, w) \big)\\
    =~& (\rho_2)_{\lambda + \mu} ( d [u_\lambda v]^1, w) + (\rho_2)_\mu (dv, [u_\lambda w]^1) \\
    =~& [[u_\lambda v]^1_{\lambda + \mu} w]^1 + [v_\mu [u_\lambda w]^1]^1.
\end{align*}
This shows that $\mathcal{G}_1$ is also a Leibniz conformal algebra with the above $\lambda$-bracket. Moreover, the $\mathbb{C}[\partial]$-linear map $d : \mathcal{G}_1 \rightarrow \mathcal{G}_0$ is a morphism of Leibniz conformal algebras as
\begin{align*}
    d [u_\lambda v]^1 = d (\rho_2)_\lambda (du, v) = (\rho_2)_\lambda (du, dv) = [(du)_\lambda (dv)]^0, \text{ for all } u, v \in \mathcal{G}_1.
\end{align*}
Furthermore, the conformal sesquilinear maps $\rho_2 : \mathcal{G}_0 \otimes \mathcal{G}_1 \rightarrow \mathcal{G}_1 [[ \lambda ]]$ and $\rho_2 : \mathcal{G}_1 \otimes \mathcal{G}_0 \rightarrow \mathcal{G}_1 [[\lambda ]]$ make the $\mathbb{C}[\partial]$-module $\mathcal{G}_1$ into a representation of the Leibniz conformal algebra $\mathcal{G}_0$  (which follows from (vi), (vii) and (viii) of Definition \ref{2termdef}). They also satisfy some additional conditions. Motivated by this, we now introduce the following notion.

\begin{definition}
A {\em crossed module of Leibniz conformal algebras} is a quintuple $(\mathfrak{g}, \mathfrak{h}, d , \Phi^l, \Phi^r)$ in which $\mathfrak{g}, \mathfrak{h}$ are both Leibniz conformal algebras, $d: \mathfrak{g} \rightarrow \mathfrak{h}$ is a morphism of Leibniz conformal algebras and there are conformal sesquilinear maps
\begin{align*}
    \Phi^l : \mathfrak{h} \otimes \mathfrak{g} \rightarrow \mathfrak{g} [[ \lambda]], ~ v \otimes x \mapsto \Phi^l_\lambda (v, x) ~~~~ \text{ and } ~~~~ \Phi^r : \mathfrak{g} \otimes \mathfrak{h} \rightarrow \mathfrak{g} [[ \lambda]], ~ x \otimes v \mapsto \Phi^r_\lambda ( x, v)
\end{align*}
that make $\mathfrak{g}$ into a representation of the Leibniz conformal algebra $\mathfrak{h}$ satisfying additionally for $x, y \in \mathfrak{g}$, $h \in \mathfrak{h}$,
\begin{align*}
    d \big( \Phi^l_\lambda (h, x)   \big) =~& [h_\lambda (dx)]^\mathfrak{h},\\
    d \big(   \Phi^r (x, h)  \big) =~& [(dx)_\lambda h]^\mathfrak{h}, \\
    \Phi^l_\lambda (dx, y) =~& [x_\lambda y]^\mathfrak{g} = \Phi^r_\lambda (x, dy),\\
    [x_\lambda \Phi^r_\mu (y,h)]^\mathfrak{g} =~& \Phi^r_{\lambda + \mu} ( [x_\lambda y]^\mathfrak{g}, h) + [y_\mu \Phi^r_\lambda (x, h)]^\mathfrak{g}, \\
    [x_\lambda \Phi^l_\mu (h, y)]^\mathfrak{g} =~& [\Phi^r_\lambda (x, h)_{\lambda + \mu} y]^\mathfrak{g} + \Phi^l_\mu (h, [x_\lambda y]^\mathfrak{g}), \\
    \Phi^l_\lambda (h, [x_\mu y]^\mathfrak{g}) =~& [\Phi_\lambda^l (h, x)_{\lambda + \mu} y]^\mathfrak{g} + [x_\mu \Phi^l_\lambda (h, y)]^\mathfrak{g}.
\end{align*}
Here $[\cdot_\lambda \cdot]^\mathfrak{g}$   (resp. $[\cdot_\lambda \cdot]^\mathfrak{h}$) denotes the $\lambda$-bracket of the Leibniz conformal algebra $\mathfrak{g}$ (resp. $\mathfrak{h}$).
\end{definition}

With the above definition, we have the following result.

\begin{thm}
    There is a bijective correspondence between strict $Leib_\infty$-conformal algebras and crossed modules of Leibniz conformal algebras.
\end{thm}

\begin{proof}
    Let $(\mathcal{G}_1 \xrightarrow{d} \mathcal{G}_0, \rho_2, \rho_3 = 0)$ be a strict $Leib_\infty$-conformal algebra. Then our previous discussions show that $(\mathcal{G}_1, \mathcal{G}_0, d , \rho_2, \rho_2)$ is a crossed module of Leibniz conformal algebras. Conversely, let $(\mathfrak{g}, \mathfrak{h}, d, \Phi^l, \Phi^r)$ be a crossed module of Leibniz conformal algebras. Then it is easy to observe that $(\mathfrak{g} \xrightarrow{d} \mathfrak{h}, \rho_2, \rho_3 = 0)$ is a strict $Leib_\infty$-conformal algebra, where
    \begin{align*}
        (\rho_2)_\lambda (u, v) = [u_\lambda v]^\mathfrak{h}, ~~~~ (\rho_2)_\lambda (u, x) = \Phi^l_\lambda (u, x) ~~ \text{ and } ~~ (\rho_2)_\lambda (x, u) = \Phi^r_\lambda (x, u), \text{ for } u, v \in \mathfrak{h}, x \in \mathfrak{g}.
    \end{align*}
    This completes the proof.
\end{proof}

\section{Leibniz conformal 2-algebras}\label{sec5} In this section, we categorify Leibniz conformal algebras and hence introduce the notion of Leibniz conformal $2$-algebras. We show that the collection of all Leibniz conformal $2$-algebras and homomorphisms between them forms a category, denoted by {\bf LeibConf2}. We show that the category {\bf LeibConf2} is equivalent to the category {\bf 2Leib$_\infty$Conf} of $2$-term $Leib_\infty$-conformal algebras.

Recall that the concept of Leibniz $2$-algebras (that is, categorified Leibniz algebras) was first introduced by Sheng and Liu \cite{sheng-liu} as the non-skew-symmetric version of Lie $2$-algebras. Here we generalize their notion in the context of conformal algebras.

Let ${\bf Vect}_{\mathbb{C}[\partial]}$ be the category of $\mathbb{C}[\partial]$-modules and homomorphisms between them. Thus, objects in ${\bf Vect}_{\mathbb{C}[\partial]}$ are $\mathbb{C}[\partial]$-modules and homomorphisms from a $\mathbb{C}[\partial]$-module $M$ to another $\mathbb{C}[\partial]$-module $N$ are precisely $\mathbb{C}[\partial]$-linear maps from $M$ to $N$.

\begin{definition}
A {\bf $2$-vector space in the category of $\mathbb{C}[\partial]$-modules} is a category internal to the category ${\bf Vect}_{\mathbb{C}[\partial]}$. 
\end{definition}

Thus, a $2$-vector space in the category of $\mathbb{C}[\partial]$-modules is a category $C = (C_1 \rightrightarrows C_0)$ in which both the collection of objects $C_0$ and the collection of morphisms $C_1$ are $\mathbb{C}[\partial]$-modules such that the source and target maps $s, t: C_1 \rightarrow C_0$, the object-inclusion map $i : C_0 \rightarrow C_1$ and the composition map $m : C_1 \times_{C_0} C_1 \rightarrow C_1$ are all $\mathbb{C}[\partial]$-linear maps. In a $2$-vector space $C = (C_1 \rightrightarrows C_0)$ in the category of $\mathbb{C}[\partial]$-modules, we write the images of the object-inclusion map as $i(x) = 1_x$, for $x \in C_0$. Let $C = (C_1 \rightrightarrows C_0)$ and $C' = (C'_1 \rightrightarrows C'_0)$ be two $2$-vector spaces in the category of $\mathbb{C}[\partial]$-modules. A {\em homomorphism} from $C$ to $C'$ is given by a pair $F = ( F_0, F_1)$ that consist of two $\mathbb{C}[\partial]$-linear maps $F_0: C_0 \rightarrow C_0'$ and $F_1 : C_1 \rightarrow C_1'$ which commute with all the structure maps, i.e. the following diagram commutes

\[
\xymatrix{
C_1 \ar[r]^s \ar[d]_{F_1} & C_0 \ar[d]^{F_0} & C_1 \ar[d]_{F_1} \ar[r]^t & C_0 \ar[d]^{F_0} & C_0 \ar[r]^i \ar[d]_{F_0} & C_1 \ar[d]^{F_1} & C_1 \times_{C_0} C_1 \ar[d]_{F_1 \times_{C_0} F_1} \ar[r]^m & C_1 \ar[d]^{F_1} \\
C'_1 \ar[r]_{s'} & C'_0 & C'_1 \ar[r]_{t'} & C'_0 & C'_0 \ar[r]_{i'} & C'_1 & C'_1 \times_{C'_0} C'_1 \ar[r]_{m'} & C'_1. 
}
\]
A homomorphism is also called a $\mathbb{C}[\partial]$-linear functor. The collection of all $2$-vector spaces in the category of $\mathbb{C}[\partial]$-modules and homomorphisms between them forms a category.

Note that $2$-vector spaces in the category of $\mathbb{C}[\partial]$-modules are closely related to $2$-term complexes of $\mathbb{C}[\partial]$-modules. Let $C=(C_{1}\rightrightarrows C_{0})$ be a $2$-vector space in the category of $\mathbb{C}[\partial]$-modules. Then $\mathrm{ker}(s)\xrightarrow{t} C_{0} $ is a $2$-term complex of $\mathbb{C}[\partial]$-modules. On the other hand, if $\mathcal{G}_{1} \xrightarrow{ d} \mathcal{G}_{0}$ is a complex of $\mathbb{C}[\partial]$-modules then
\begin{align}\label{2-to-2}
    C = ( \mathcal{G}_{0}\oplus \mathcal{G}_{1}\rightrightarrows \mathcal{G}_{0})
\end{align}
is a $2$-vector space in the category of $\mathbb{C}[\partial]$-modules, where the source, target and object inclusion are respectively given by 
\begin{align*}
    s(x,h) =x, \quad t(x,h)= x+ d h \quad \mathrm{ and } \quad i(x) = (x,0),
\end{align*}
for $(x, h) \in \mathcal{G}_0 \oplus \mathcal{G}_1$ and $x \in \mathcal{G}_0$. A morphism between $2$-term complexes of $\mathbb{C}[\partial]$-modules induces a morphism between the corresponding $2$-vector spaces in the category of $\mathbb{C}[\partial]$-modules.

\begin{definition}
    A {\em Leibniz conformal $2$-algebra} is a triple $(C, [ \cdot_\lambda \cdot ], \mathcal{L})$ consisting of a $2$-vector space $C = (C_1 \rightrightarrows C_0)$ in the category of $\mathbb{C}[\partial]$-modules, a $\mathbb{C}$-linear conformal sesquilinear functor $[ \cdot_\lambda \cdot ]: C \otimes C \rightarrow C[[\lambda ]]$ and a conformal sesquilinear natural isomorphism (called the conformal Leibnizator)
   \begin{align*}
       \mathcal{L}_{x, y, z} : [x_\lambda [y_\mu z ]] \rightarrow  [[x_\lambda y]_{\lambda + \mu} z] + [y_\mu [x_\lambda z]]
   \end{align*}
   such that for any $x, y, z, w \in C_0$, the following diagram is  commutative
   \begin{align}\label{hexa}
   \xymatrix{
    &  [x_\lambda [y_\mu [z_\nu w ]]]
    \ar[ld]_{ [x_\lambda \mathcal{L}_{y, z, w}]} \ar[rd]^{\mathcal{L}_{x, y, [z_\nu w]}} &  \\
    \resizebox{0.23\textwidth}{!}{$[x_\lambda [[y_\mu z]_{\mu + \nu} w]] + [x_\lambda [z_\nu [y_\mu w]]]$} \ar[d]^{\mathcal{L}_{x, [y_\mu z], w} + \mathcal{L}_{x, z, [y_\lambda w]}} & & \resizebox{0.23\textwidth}{!}{$[[x_\lambda y]_{\lambda + \mu} [z_\nu w]] + [y_\mu [x_\lambda [z_\nu w ]]] $} \ar[d]^{ 1 + [y_\mu \mathcal{L}_{x, z, w}]} \\
    \substack{[[x_\lambda [y_\mu z]]_{\lambda + \mu + \nu} w] + [[y_\mu z]_{\mu + \nu} [x_\lambda w]] \\ + [[x_\lambda z]_{\lambda + \nu} [y_\mu w]] + [z_\nu [x_\lambda [y_\mu w]]]} \ar[rd]_{  [ (\mathcal{L}_{x, y, z})_{\lambda + \mu + \nu} w]+ 1 + 1 + [z_\nu \mathcal{L}_{x, y, w}]  } 
    & & \substack{ [[x_\lambda y]_{\lambda + \mu} [z_\nu w]] +  [y_\mu [[x_\lambda z]_{\lambda + \nu} w]] \\ + [y_\mu [z_\nu [x_\lambda w]]] } \ar[ld]^{  \mathcal{L}_{[x_\lambda y], z, w + \mathcal{L}_{y, [x_\lambda z], w } + \mathcal{L}_{y, z, [x_\lambda w] }   } }  \\
     & \substack{[[[x_\lambda y]_{\lambda + \mu} z]_{\lambda + \mu + \nu} w] + [z_\nu [[x_\lambda y]_{\lambda + \mu} w]] \\ + [[y_\mu [x_\lambda z]]_{\lambda + \mu + \nu} w] + [[x_\lambda z]_{\lambda + \mu} [y_\mu w]] \\
    +  [[y_\mu z]_{\mu+ \nu} [x_\lambda w]] + [z_\nu [y_\mu [x_\lambda w]]] .}  & 
   }
   \end{align}
\end{definition}

\medskip

In a Leibniz conformal $2$-algebra, if the conformal sesquilinear functor $[\cdot_\lambda \cdot ]$ and the conformal Leibnizator are both skew-symmetric, one recovers the notion of a Lie conformal $2$-algebra \cite{tao}. Therefore, Leibniz conformal $2$-algebras can be realized as the non-skew-symmetric analogue of Lie conformal $2$-algebras. 

Let $(C, [\cdot_\lambda \cdot], \mathcal{L})$ and $(C', [\cdot_\lambda \cdot]', \mathcal{L}')$ be two Leibniz conformal $2$-algebras. A {\em homomorphism} of Leibniz conformal $2$-algebras is a conformal functor $F = (F_0, F_1) : C \rightarrow C'$ between the underlying $2$-vector spaces (in the category of $\mathbb{C}[\partial]$-modules) and a natural isomorphism
\begin{align*}
    (F_2)^{x, y} : [F_0 (x)_\lambda F_0 (y)]' \rightarrow F_0 [x_\lambda y]
\end{align*}
such that for any $x, y, z \in C_0$, the following diagram is commutative
\[
\xymatrix{
[F_0 (x)_\lambda [F_0 (y)_\mu F_0 (z)]']' \ar[r]^{\mathcal{L}'_{F_0 (x), F_0 (y), F_0 (z)} } \ar[d]_{[F_0 (x)_\lambda (F_2)^{y,z}]'} &  [[F_0 (x)_\lambda F_0 (y)]'_{\lambda + \mu} F_0 (z)]' + [F_0 (y)_\mu [F_0 (x)_\lambda F_0 (z)]']' \ar[d]^{ [ ((F_2)^{x, y} )_{\lambda + \mu} F_0 (z)]' + [F_0 (y)_\mu (F_2)^{x, z}]'  } \\
[F_0 (x)_\lambda F_0 [y_\mu z]]' \ar[d]_{ (F_2)^{x, [y_\mu z]} } &  [F_0 [x_\lambda y]_{\lambda + \mu} F_0 (z)]' + [F_0 (y)_\mu F_0 [x_\lambda z]]' \ar[d]^{ (F_2)^{[x_\lambda y], z} + (F_2)^{y, [x_\lambda z]}} \\
F_0 [x_\lambda [y_\mu z]] \ar[r]_{F_0 (\mathcal{L}_{x, y, z})} & F_0 ( [[x_\lambda y]_{\lambda + \mu} z] + [y_\mu [x_\lambda z]] ). \\
}
\]
We often denote a homomorphism as above by the pair $(F = (F_0, F_1), F_2)$.

\medskip

Let $(C, [\cdot_\lambda \cdot], \mathcal{L}),~(C', [\cdot_\lambda \cdot]', \mathcal{L}')$ and $(C'', [\cdot_\lambda \cdot]'', \mathcal{L}'')$ be three Leibniz conformal $2$-algebras. Suppose $(F= (F_0, F_1), F_2) : (C, [\cdot_\lambda \cdot], \mathcal{L}) \rightarrow (C', [\cdot_\lambda \cdot]', \mathcal{L}')$ and $(G= (G_0, G_1), G_2) : (C', [\cdot_\lambda \cdot]', \mathcal{L}') \rightarrow (C'', [\cdot_\lambda \cdot]'', \mathcal{L}'')$ are homomorphisms between Leibniz conformal $2$-algebras. Then their composition is given by $\big( G \circ F = (G_0 \circ F_0, G_1 \circ F_1), (G \circ F)_2 \big)$, where $(G \circ F)_2$ is the  ordinary composition
\begin{align*}
[(G_0 \circ F_0)(x)_\lambda (G_0 \circ F_0) (y)]'' \xrightarrow{   
 (G_2)^{  F_0 (x), F_0 (x)}} G_0 [F_0 (x)_\lambda F_0 (y)]' \xrightarrow{ G_0 \big(  (F_2)^{x, y}  \big)} (G_0 \circ F_0) [x_\lambda y], \text{ for } x, y \in C_0.
\end{align*}
Further, for any Leibniz conformal $2$-algebra $(C, [\cdot_\lambda \cdot], \mathcal{L})$, the identity homomorphism is given by the identity functor $\mathrm{id}_C = (\mathrm{id}_{C_0}, \mathrm{id}_{C_1}): C \rightarrow C$ with the identity natural isomorphism $(\mathrm{id}_C)_2^{x, y} : [x_\lambda y] \rightarrow [x_\lambda y]$, for any $x, y \in C_0$. With all the above definitions and notations, the collections of all Leibniz conformal $2$-algebras and homomorphisms between them form a category, denoted by ${\bf LeibConf2}$.

In the following result, we give an equivalent description of the category ${\bf LeibConf2}$. More precisely, we have the following.

\begin{thm}
    The category {\bf LeibConf2} is equivalent to the category {\bf 2Leib$_\infty$Conf}.
\end{thm}

\begin{proof}
    Let $(C = (C_1 \rightrightarrows C_0), [\cdot_\lambda \cdot], \mathcal{L})$ be a Leibniz conformal $2$-algebra. Consider the $2$-term complex $\mathcal{G}_1 = \mathrm{ker}(s) \xrightarrow{ t } C_0 = \mathcal{G}_0$ of $\mathbb{C}[\partial]$-modules. We define conformal sesquilinear maps $\rho_2 : \mathcal{G}_i \otimes \mathcal{G}_j \rightarrow \mathcal{G}_{i+j} [[\lambda ]]$ and $\rho_3 : \mathcal{G}_0 \otimes \mathcal{G}_0 \otimes \mathcal{G}_0 \rightarrow \mathcal{G}_1 [[\lambda, \mu ]]$ by
    \begin{align*}
        (\rho_2)_\lambda (x, y) = [x_\lambda y], \quad (\rho_2) (x, h) = [(1_x)_\lambda h], \quad (\rho_2)_\lambda (h, x) = [h_\lambda (1_x)], \\
        (\rho_3)_{\lambda, \mu} (x, y, z) = {pr_2} \big(  \mathcal{L}_{x, y, z} : [x_\lambda [y_\mu z ]] \rightarrow  [[x_\lambda y]_{\lambda + \mu} z] + [y_\mu [x_\lambda z]]  \big),
        \end{align*}
        for $x, y, z \in \mathcal{G}_0$ and $h \in \mathcal{G}_1$. Here $pr_2 : (\mathcal{G}_0 \oplus \mathcal{G}_1) [[\lambda, \mu ]] \rightarrow \mathcal{G}_1 [[\lambda, \mu ]]$ is the projection onto the second factor. Then {it is easy to verify that} $\big(  \mathcal{G}_1 = \mathrm{ker}(s) \xrightarrow{t} C_0 = \mathcal{G}_0, \rho_2, \rho_3 \big)$ is a $2$-term $Leib_\infty$-conformal algebra.

        Next, let $(C, [\cdot_\lambda \cdot], \mathcal{L})$ and $(C', [\cdot_\lambda \cdot]', \mathcal{L}')$ be two Leibniz conformal $2$-algebras and $(F = (F_0, F_1), F_2)$ be a homomorphism between them. We consider the corresponding $2$-term $Leib_\infty$-conformal algebras
        $\big(  \mathcal{G}_1 = \mathrm{ker}(s) \xrightarrow{t} C_0 = \mathcal{G}_0, \rho_2, \rho_3 \big)$ and $\big(  \mathcal{G}'_1 = \mathrm{ker}(s') \xrightarrow{t} C'_0 = \mathcal{G}'_0, \rho'_2, \rho'_3 \big)$, respectively. We define $\mathbb{C}[\partial]$-module maps $f_0 : C_0 \rightarrow C_0'$, $f_1 : \mathrm{ker}(s) \rightarrow \mathrm{ker} (s')$ and a conformal sesquilinear map $f_2 : C_0 \otimes C_0 \rightarrow \mathrm{ker}(s') [[\lambda ]]$ by $f_0 = F_0$, $f_1 = F_1 |_{\mathrm{ker}(s)}$ and $f_2 (x, y) = (F_2 )^{x, y} - 1_{s ( (F_2)^{x, y}) }$, for $x, y \in C_0$. Then it turns out that $(f_0, f_1, f_2)$ defines a homomorphism between the above $2$-term $Leib_\infty$-conformal algebras. Therefore, we obtain a functor $S: {\bf LeibConf2} \rightarrow {\bf 2Leib_\infty Conf}$.

\medskip

        Conversely, let $(\mathcal{G}_1 \xrightarrow{d} \mathcal{G}_0, \rho_2, \rho_3)$ be a $2$-term $Leib_\infty$-conformal algebra. First, we consider the $2$-vector space $C = (\mathcal{G}_0 \oplus \mathcal{G}_1 \rightrightarrows \mathcal{G}_0)$ in the category of $\mathbb{C}[\partial]$-modules. We now define a conformal sesquilinear functor $[\cdot_\lambda \cdot] : C \otimes C \rightarrow C [[\lambda ]]$ by
        \begin{align*}
            [(x, h)_\lambda (y, k)] := \big(  (\rho_2)_\lambda (x, y), ~ (\rho_2)_\lambda (x, k) + (\rho_2)_\lambda (h, y) + (\rho_2)_\lambda (dh, k) \big),
        \end{align*}
        for $(x, h), (y, k) \in C_1 = \mathcal{G}_0 \oplus \mathcal{G}_1$. We define the conformal Leibnizator by
        \begin{align}\label{conf-leibnizator}
            \mathcal{L}_{x, y, z} := \big( [x_\lambda [y_\mu z]] , - (\rho_3)_{\lambda, \mu} (x, y, z) \big), \text{ for } x, y, z \in \mathcal{G}_0.
        \end{align}
        It follows from the condition (v) of Definition \ref{2termdef} that the conformal Leibnizator $\mathcal{L}_{x, y, z}$ has the source $[x_\lambda [y_\mu z]]$ and the target $[[x_\lambda y]_{\lambda + \mu} z] + [y_\mu [x_\lambda z]]$, for $x, y , z \in \mathcal{G}_0$. It is also easy to check that $\mathcal{L}_{x, y, z}$ is a natural isomorphism. Finally, a direct computation shows that the conformal Leibnizator (\ref{conf-leibnizator}) makes the diagram (\ref{hexa}) commutative. Thus, we obtain a Leibniz conformal $2$-algebra $(C = (\mathcal{G}_0 \oplus \mathcal{G}_1 \rightrightarrows \mathcal{G}_0), [\cdot_\lambda \cdot], \mathcal{L})$.

        Next, suppose $(\mathcal{G}_1 \xrightarrow{d} \mathcal{G}_0, \rho_2, \rho_3)$ and $(\mathcal{G}'_1 \xrightarrow{d'} \mathcal{G}'_0, \rho'_2, \rho'_3)$ are $2$-term $Leib_\infty$-conformal algebras, and say $(f_0, f_1, f_2)$ is a homomorphism between them. Consider the corresponding Leibniz conformal $2$-algebras $\big(  C = (\mathcal{G}_0 \oplus \mathcal{G}_1 \rightrightarrows \mathcal{G}_0), [\cdot_\lambda \cdot], \mathcal{L}  \big)$ and $\big(  C' = (\mathcal{G}'_0 \oplus \mathcal{G}'_1 \rightrightarrows \mathcal{G}'_0), [\cdot_\lambda \cdot]', \mathcal{L}'  \big)$. We define a $\mathbb{C}[\partial]$-linear functor $F = (F_0, F_1) : C \rightarrow C'$ by 
        \begin{align*}
            F_0 = f_0 : \mathcal{G}_0 \rightarrow \mathcal{G}_0' ~~~ \text{ and } ~~~ F_1 = f_0 \oplus f_1 : \mathcal{G}_0 \oplus \mathcal{G}_1 \rightarrow \mathcal{G}_0' \oplus \mathcal{G}_1', 
        \end{align*}
        and a natural isomorphism $(F_2)^{x, y} : [F_0 (x)_\lambda F_0 (y)]' \rightarrow F_0 [x_\lambda y]$ by $(F_2)^{x, y} = \big( [f_0 (x)_\lambda f_0 (y)]', - (f_2)_\lambda (x, y) \big)$, for $x, y \in \mathcal{G}_0$. Then a straightforward calculation shows that $(F = (F_0, F_1), F_2)$ is a homomorphism between the above Leibniz conformal $2$-algebras. Hence, we obtain a functor $T: {\bf 2Leib_\infty Conf} \rightarrow {\bf LeibConf2}$.

        \medskip

        In the following, we construct natural isomorphisms $\alpha : T \circ S \Rightarrow 1_{\bf LeibConf2}$ and $\beta : S \circ T \Rightarrow 1_{\bf 2Leib_\infty Conf}$ that will show the equivalence of the categories ${\bf LeibConf2}$ and ${\bf 2Leib_\infty Conf}$. Let $(C, [\cdot_\lambda \cdot], \mathcal{L})$ be a Leibniz conformal $2$-algebra. If we apply the functor $S$, we obtain the $2$-term $Leib_\infty$-conformal algebra $(\mathrm{ker}(s) \xrightarrow{ t} C_0, \rho_2, \rho_3)$. Further, if we apply the functor $T$, we obtain a new Leibniz conformal $2$-algebra, say $\big(  C' = (C_0 \oplus \mathrm{ker}(s) \rightrightarrows C_0), [\cdot_\lambda \cdot]', \mathcal{L}' \big)$.
        We define a $\mathbb{C}[\partial]$-linear functor $\alpha_C = \big(  (\alpha_C)_0, (\alpha_C)_1  \big) : C' \rightarrow C$  by $(\alpha_C)_0 (x) = x$ and $(\alpha_C)_1 (x, h) = 1_x + h$, for $x \in C_0$, $(x, h) \in C_0 \oplus \mathrm{ker}(s)$. There is also the identity natural isomorphism 
        \begin{align*}
        (\alpha_C)^{x, y}_2 = \mathrm{id} : [(\alpha_c)_0 (x)_\lambda (\alpha_C)_0 (y)] = [x_\lambda y] \rightarrow (\alpha_C)_0 [x_\lambda y]' = [x_\lambda y].
        \end{align*}
        It is also easy to see that $\big( \alpha_C = \big( (\alpha_C)_0, (\alpha_C)_1 \big) , (\alpha_C)_2  \big)$ is an {isomorphism} of Leibniz conformal $2$-algebras from $(C', [\cdot_\lambda \cdot]', \mathcal{L}')$ to $(C, [\cdot_\lambda \cdot], \mathcal{L})$. This construction yields a natural isomorphism $\alpha: T \circ S \Rightarrow 1_{\bf LeibConf2}$. Next, let $(\mathcal{G}_1 \xrightarrow{d} \mathcal{G}_0, \rho_2, \rho_3)$ be a $2$-term $Leib_\infty$-conformal algebra. If we apply $S \circ T$, we get back the same $2$-term $Leib_\infty$-conformal algebra. Therefore, we may take the natural isomorphism $\beta: S \circ T \Rightarrow 1_{\bf 2Leib_\infty Conf}$ to be the identity natural isomorphism. This completes the proof.
\end{proof}

\subsection{An example of a Leibniz conformal 2-algebra} Let $A$ and $B$ be two finite $\mathbb{C} [\partial]$-modules. A {\em conformal linear map} from $A$ to $B$ is a $\mathbb{C}$-linear map $f: A \rightarrow B [\lambda], x \mapsto f_\lambda (x)$ that satisfies $f_\lambda (\partial x) = (\partial + \lambda) f_\lambda (x)$, for all $x \in A$. We denote the space of all conformal linear maps from $A$ to $B$ by the notation $\mathrm{CHom}(A, B)$. The space $\mathrm{CHom}(A, B)$ has a canonical $\mathbb{C}[\partial]$-module structure with the action of $\partial$ given by $(\partial f)_\lambda = - \lambda f_\lambda$, for all $f \in \mathrm{CHom}(A, B)$.

We use the notation $\mathrm{CEnd} (A)$ for the space $\mathrm{CHom} (A, A)$ of all conformal linear maps from $A$ to $A$. Then there is a $\lambda$-product $\cdot_\lambda \cdot : \mathrm{CEnd} (A) \otimes \mathrm{CEnd} (A) \rightarrow \mathrm{CEnd} (A) [\lambda]$ given by
\begin{align*}
    (f_\lambda g)_\mu x := f_\lambda (g_{\mu - \lambda} (x)), \text{ for } f, g \in \mathrm{CEnd} (A), x \in A
\end{align*}
which makes $\mathrm{CEnd} (A)$ into an associative conformal algebra in the sense of \cite{bkv}. As a consequence, $(\mathrm{CEnd} (A), [\cdot_\lambda \cdot ] )$ is a Lie conformal algebra, where the $\lambda$-bracket is given by
\begin{align*}
    [f_\lambda g] := f_\lambda g - g_{ - \partial - \lambda} f, \text{ for } f, g \in \mathrm{CEnd}(A).
\end{align*}
It is easy to see that the Lie conformal algebra $(\mathrm{CEnd} (A), [\cdot_\lambda \cdot ] )$ has a representation on the $\mathbb{C}[\partial]$-module $A$ with the $\lambda$-action $f_\lambda x := f_\lambda (x)$, for $f \in \mathrm{CEnd} (A)$ and $x \in A$. Thus, the direct sum $\mathbb{C}[\partial]$-module $\mathrm{CEnd} (A) \oplus A$ carries a Leibniz conformal algebra structure the $\lambda$-bracket
\begin{align*}
    \{ (f, x)_\lambda (g, y) \} = ([f_\lambda g], f_\lambda y), \text{ for } (f, x), (g, y) \in \mathrm{CEnd} (A) \oplus A.
\end{align*}

In the following, we will generalize the above construction in the graded context. Let $A: (A_1 \xrightarrow{ d} A_0)$ be a complex of finite $\mathbb{C}[\partial]$-modules. We define
\begin{align*}
    \mathrm{CEnd}^0 (A) = \{ (f_0, f_1) |~& f_0 : A \rightarrow A_0 [\lambda] \text{ and } f_1 : A_1 \rightarrow A_1 [\lambda] \text{ are} \\ & \text{ conformal linear maps with } d \circ f_1 = f_0 \circ d  \}
\end{align*}
and $\mathrm{CEnd}^1 (A) = \mathrm{CHom} (A_0, A_1)$. Note that both the spaces $\mathrm{CEnd}^0 (A)$ and $\mathrm{CEnd}^1 (A)$ has obvious $\mathbb{C}[\partial]$-module structures. Then the map $\Delta : \mathrm{CEnd}^1 (A) \rightarrow \mathrm{CEnd}^0 (A)$ given by $\Delta (h) = (d \circ h, h \circ d)$ is a $\mathbb{C}[\partial]$-module map. Moreover, the graded $\mathbb{C}[\partial]$-module $\mathrm{CEnd} (A) := \mathrm{CEnd}^0 (A) \oplus \mathrm{CEnd}^1 (A)$ carries a graded Lie conformal algebra structure with the $\lambda$-bracket
\begin{align*}
    &[(f_0, f_1)_\lambda (g_0, g_1)] := \big(  (f_0)_\lambda (g_0) - (g_0)_{- \partial - \lambda} (f_0),   (f_1)_\lambda (g_1) - (g_1)_{- \partial - \lambda} (f_1) \big),\\
    &[(f_0, f_1)_\lambda h] = - [h_{-\partial - \lambda} (f_0, f_1)] := (f_1)_\lambda h - h_{-\partial - \lambda} (f_0),
\end{align*}
for $(f_0, f_1), (g_0, g_1) \in \mathrm{CEnd}^0 (A)$ and $h \in \mathrm{CEnd}^1 (A)$. With the above operations, $(   \mathrm{CEnd} (A), \Delta, [\cdot_\lambda \cdot])$ is a differential graded Lie conformal algebra. Further, it is easy to verify that the differential graded Lie conformal algebra $(   \mathrm{CEnd} (A), \Delta, [\cdot_\lambda \cdot])$  has a representation on the graded $\mathbb{C}[\partial]$-module $A$ with the $\lambda$-action given by
\begin{align*}
    (f_0, f_1)_\lambda x = (f_0)_\lambda x, \quad (f_0, f_1)_\lambda v = (f_1)_\lambda v, \quad h_\lambda x = h_\lambda (x), \quad  h_\lambda v = 0,
\end{align*}
for $(f_0, f_1) \in \mathrm{CEnd}^0 (A)$, $h \in \mathrm{CEnd}^1 (A)$, $x \in A_0$ and $v \in A_1$. As a consequence, the direct sum 
\begin{align*}
\mathrm{CEnd}(A) \oplus A = \big(  \mathrm{CEnd}^0 (A) \oplus A_0  \big) \oplus \big( \mathrm{CEnd}^1 (A) \oplus A_1 \big) 
\end{align*}
carries a differential graded Leibniz conformal algebra structure with the differential given by $\Delta + d: \mathrm{CEnd}^1 (A) \oplus A_1 \rightarrow \mathrm{CEnd}^0 (A) \oplus A_0$ and the $\lambda$-bracket
\begin{align*}
\big\{   \big(  (f_0, f_1) , x \big)_\lambda  \big(  (g_0, g_1) , y \big) \big\} =~&  \big(  [   (f_0, f_1)_\lambda (g_0, g_1)  ], (f_0, f_1)_\lambda y  \big),\\
\big\{  \big(  (f_0, f_1) , x \big)_\lambda (h, v) \big\} =~&  \big(  [ (f_0, f_1)_\lambda h], (f_0, f_1)_\lambda v  \big),\\
\big\{  (h, y)_\lambda \big(  (f_0, f_1) , x \big) \big\} =~&  \big( [h_\lambda (f_0, f_1)], h_\lambda x \big),
\end{align*}
for any $\big(  (f_0, f_1) , x \big), \big(  (g_0, g_1) , y \big) \in \mathrm{CEnd}^0 (A) \oplus A_0$ and $(h, y) \in \mathrm{CEnd}^1 (A) \oplus A_1$. Since it is a $2$-term differential graded Leibniz conformal algebra, it can be considered as a $2$-term $Leib_\infty$-algebra with vanishing $\rho_3.$ Hence it is an example of a strict $Leib_\infty$-conformal algebra. Therefore, it can be realized as a Leibniz conformal $2$-algebra with the trivial conformal Leibnizator.

\medskip

\medskip

\noindent {\bf Conflict of interest statement.} There is no conflict of interest.

\medskip

\noindent {\bf Data Availability Statement.} Data sharing does not apply to this article as no new data were created or analyzed in this study.

\medskip

\noindent {\bf Acknowledgements.} 
Anupam Sahoo would like to thank CSIR, Government of India for funding the PhD fellowship. Both authors thank the Department of Mathematics, IIT Kharagpur for providing the beautiful academic atmosphere where the research has been carried out.

\end{document}